\documentclass[preprint]{imsart}
\RequirePackage[OT1]{fontenc}
\RequirePackage{amsthm,amsmath}
\RequirePackage[colorlinks,citecolor=blue,urlcolor=blue,backref=true,linktocpage=true]{hyperref}
\startlocaldefs
\numberwithin{equation}{section}
\theoremstyle{plain}
\endlocaldefs

\usepackage{amsmath}
\usepackage{amsmath, amsthm, amssymb}
\usepackage{accents}
\usepackage{graphicx}
\usepackage{color}
\usepackage{enumerate}
\usepackage{textcomp}
\usepackage{graphicx}
\usepackage{graphicx,hyperref}
\usepackage{hyperref}
\usepackage{natbib}
\usepackage{fancyhdr}
\usepackage[left=1.3in, right=1.3in]{geometry}

\newtheorem{thm}{Theorem}[section]
\newtheorem{lem}{Lemma}[section]

\theoremstyle{definition}
\newtheorem{definition}{Definition}[section]

\newtheorem{assumption}{Assumption}[section]
\newtheorem{remark}{Remark}[section]

\numberwithin{equation}{section}

\begin{document}

\begin{frontmatter}
\title{Asymptotic Minimaxity, Optimal Posterior Concentration and Asymptotic Bayes Optimality of Horseshoe-type Priors Under Sparsity}

% \title{Asymptotic Minimaxity, Optimal Posterior Concentration and Asymptotic Bayes Optimality Under Sparsity of Horseshoe-type Priors in Some Sparse High Dimensional Problems}
% On Some Optimality Properties of a General Class of Shrinkage Priors in Some Sparse High Dimensional Problems}
\runtitle{Asymptotic Minimaxity and ABOS of Horseshoe-type Priors Under Sparsity}
% \runtitle{Asymptotic Optimality of Shrinkage Priors in High-dimensional Problems}
%  
% \thankstext{T1}{Footnote to the title with the ``thankstext'' command.}

\begin{aug}
\author{\fnms{Prasenjit} \snm{Ghosh}\ead[label=e1]{prasenjit$\_$r@isical.ac.in}}
\and
\author{\fnms{Arijit} \snm{Chakrabarti}\ead[label=e2]{arc@isical.ac.in}}

% \thankstext{t1}{Some comment}
% \thankstext{t2}{First supporter of the project}
%  \thankstext{t3}{Second supporter of the project}
\runauthor{P. Ghosh and A. Chakrabarti}

\affiliation{Indian Statistical Institute}

\address{Applied Statistics Unit, Indian Statistical Institute, Kolkata, India.\\
\printead{e1},
% \phantom{E-mail:prasenjit$\_$r@isical.ac.in\}
\printead*{e2}}

% \address{Address of the Third author\\
% Usually a few lines long\\
% Usually a few lines long\\
% \printead{e3}\\
% \printead{u1}}
\end{aug}

\begin{abstract}
In this article, we consider a general class of global-local scale mixtures of normals that is rich enough to include a wide variety of one-group shrinkage priors, such as the three parameter beta normal mixtures, the generalized double Pareto priors, and in particular, the horseshoe prior, the Strawderman$-$Berger prior and the normal$-$exponential$-$gamma priors. Asymptotic minimaxity of Bayes estimates under the $l_2$ norm and posterior concentration properties corresponding to this general class of priors are studied when the data is assumed to be generated according to a multivariate normal distribution with a fixed unknown mean vector that is assumed to be sparse. Assuming that the number of non-zero means is known, we show that Bayes estimators based on this general class of shrinkage priors attain the minimax risk in the $l_2$ norm (up to some multiplicative constants). In particular, for the ``horseshoe-type'' priors, such as the horseshoe, the Strawderman$-$Berger and the standard double Pareto priors, the corresponding Bayes estimates are shown to attain the minimax $l_2$ risk asymptotically up to the correct constant. Moreover, it is shown that posterior distributions based on the chosen class of priors contract around the true mean vector at the minimax $l_2$ rate and around the corresponding Bayes estimates at least as fast as the minimax rate for a wide range of choices of the global shrinkage parameter. We also provide a lower found to the total posterior spread corresponding to an important sub-family of the ``horseshoe-type'' priors, which gives important insight regarding the choice of the global shrinkage parameter for this sub-family of priors. This part of this article is inspired by the work of \citet{PKV2014}. We come up with novel unifying arguments and extend their results over the general class of one-group priors under study. In the process, we settle a conjecture made in \citet{BPPD2012} regarding such optimal posterior contraction properties of the normal$-$exponential$-$gamma and the generalized double Pareto priors.\newline

In the second half of this article, we focus on the problem of simultaneous testing for the means of independent normal observations. We consider a Bayesian decision theoretic framework, where the overall loss is assumed to be the number of misclassified hypotheses and the data is modeled through a two-groups normal mixture distribution. We adopt the asymptotic framework of \cite{BCFG2011} who introduced the notion of asymptotic Bayes optimality under sparsity (ABOS) in multiple hypothesis testing. The multiple testing procedures under study are induced by the general class of shrinkage priors considered in this paper and were also considered in \citet{GTGC2015}. Using one key technique essential for proving the aforesaid asymptotic minimaxity result, we obtain exact asymptotic expressions of the Bayes risks of such induced decisions. A remarkable consequence of this is that such induced testing procedures based on the ``horseshoe-type'' priors are asymptotically Bayes optimal under sparsity (ABOS). This, to the best of our knowledge, is the first result in the literature where in a sparse problem, a class of one-group prior distributions is shown to achieve asymptotically the optimal two-groups solution up to the correct constant. 
\end{abstract}

\begin{keyword}
\kwd{Posterior concentration}
\kwd{minimaxity}
\kwd{ABOS}
\kwd{three parameter beta}
\kwd{horseshoe}
\kwd{normal-exponential-gamma}
\kwd{Strawderman$-$Berger}
\kwd{generalized double Pareto}
\end{keyword}
\end{frontmatter}

\section[Introduction]{Introduction}
With rapid advancements in modern technology and computing facilities, high throughput data have become common in real life problems across diverse scientific fields such as genomics, biology, medicine, cosmology, finance, economics and climate studies etc. As a result, inferential problems involving a large number of unknown parameters are coming to the fore. Problems where the number of unknown parameters grows as least as fast as the number of observations are typically called high-dimensional. In such problems, often times only a few of these parameters are of real importance. For example, in a high-dimensional regression problem, the proportion of non-zero regressors or regressors with non-zero coefficients is often small compared to the total number of candidate regressors. This is called the phenomenon of sparsity. A common Bayesian approach to model data of this kind is to use a two-component point mass mixture prior for the parameters under consideration. Such priors put a positive mass at zero to model the null entries and a continuous distribution to identify the non-null effects. These are also referred to as ``spike and slab priors'' or two-groups priors. This is a very natural way of modeling data of this kind from a Bayesian view point. See \citet{JS2004} and \citet{Efron2004} in this context.\newline

Use of the two-groups prior, although very natural, often poses a very daunting task computationally, especially in high-dimensional problems. Sometimes it is also possible that most of the parameters are very close to zero, but not exactly equal to zero. So in such a case a continuous prior may be able to capture sparsity in a more flexible manner. Due to these reasons, significant efforts have gone into modeling sparse high-dimensional data in recent times through hierarchical one-group continuous priors, which are also called one-group shrinkage priors. Bayesian analysis for such one-group priors are computationally much more tractable compared to the two-groups priors using MCMC techniques. A great variety of such one-group continuous shrinkage priors have appeared in the literature over the years. Some notable early examples are the $t$-prior in \citet{Tipping2001}, the double$-$exponential prior in \citet{PC2008} and \citet{Hans2009}, and the normal$-$exponential$-$gamma priors in \citet{GB2005}. \citet{CPS2009, CPS2010} introduced the horseshoe prior, which has very appealing properties. Subsequently, many other one-group priors have been proposed in the literature, e.g, in \citet{PS2011, PS2012}, \citet{ADC2011}, \citet{ADL2012} and \citet{GB2010, GB2012, GB2013}. The class of ``three parameter beta normal'' mixture priors was introduced in \citet{ADC2011}, while the ``generalized double Pareto'' class of priors was introduced by \citet{ADL2012}. The three parameter beta normal mixture family of priors contains among others the horseshoe, the Strawderman$-$Berger and the normal$-$exponential$-$gamma priors. A different class of one-group shrinkage priors referred to as the Dirichlet$-$Laplace (DL) priors has been introduced in \citet{BPPD2012, BPPD2014}. Almost all such one-group priors can be expressed as multivariate scale-mixtures of normals by employing two levels of parameters to express the prior variances for the parameters under consideration, referred to as the ``local shrinkage parameters'' and a ``global shrinkage parameter''. While the global shrinkage parameter accounts for the overall sparsity in the data, the local shrinkage parameters are helpful in detecting the obvious signals. Many of these priors are capable of handling sparsity by assigning a significant chunk of probability around zero, while at the same time they have tails which are heavy enough to ensure a priori large probabilities for the occurrence of large signals. As a result, noise observations are shrunk towards zero, while large observations almost remain unshrunk by such priors. The latter property is referred to as the ``tail robustness'' property and the corresponding one-group priors are called ``tail robust'' priors.\newline

In recent times, researchers have started to investigate various optimality properties of estimators and testing rules based on one-group shrinkage priors, where the unknown mean or coefficient vectors are modeled through such hierarchical  one-group formulation. \citet{DG2013} showed a near oracle optimality property of multiple testing rules due to \citet{CPS2010} based on the horseshoe prior in the context of multiple testing. \citet{GTGC2015} extended the work of \citet{DG2013} for a broad class of one-group tail robust shrinkage priors that includes the horseshoe in particular. \citet{BPPD2012, BPPD2014} showed that for the estimation of a sparse multivariate normal mean vector under the quadratic loss function, the posterior arising from their proposed Dirichlet$-$Laplace prior contracts around the true mean vector at the minimax optimal rate with respect to the $l_2-$norm for an appropriate choice of the underlying Dirichlet concentration parameter. In a recent article, \citet{PKV2014} showed that for the recovery of a sparse normal mean vector, the horseshoe estimator asymptotically achieves the minimax risk with respect to the $l_2$ loss up to a multiplicative constant and the corresponding posterior distribution contracts around the true mean vector at this minimax optimal rate. In addition, they showed that the posterior distribution based on the horseshoe prior contracts around the horseshoe estimator at least as fast as the minimax rate. See \citet{BRT2009} and \citet{CSHV2014} where such issues were investigated for the Bayesian Lasso prior. \newline

It is worth mentioning in this context that studying the concentration properties of posterior distributions has several important aspects in Bayesian inference. {\it Firstly}, it is often of interest to know whether a posterior distribution puts adequate mass around the neighborhood of the true distribution. For the sparse normal means model such as those considered in \citet{BPPD2014} and \citet{PKV2014}, one typically needs to show that the resulting posterior distribution puts probability mass 1 on balls centered around the true mean vector and having square radius proportional to the minimax risk under the $l_2$ norm. {\it Secondly}, for realistic uncertainty quantification, it is necessary that a posterior distribution contracts around its center at the same rate at which it converges to the true parameter value. See \citet{PKV2014} in this context. As commented in \citet{CV2012} that although construction of prior distributions in Bayesian inference is not driven by the ultimate goal of producing posterior distributions which have optimal concentration properties in terms of the frequentist minimax risk, for theoretical investigations, the minimax rate can be taken as a benchmark.\newline

A question which is therefore natural to ask and also posed in \citet{PKV2014} is what aspects of one-group shrinkage priors are essential towards attaining optimal posterior concentration properties. \citet{PKV2014} also raised an important question that whether a sharp pick at the origin, combined with a thick tail, is essential for achieving such optimal concentration properties. \citet{BPPD2012} conjectured that heavy tailed prior distributions, such as the horseshoe, the normal$-$exponential$-$gamma and the generalized double Pareto priors, should possess optimal posterior contraction properties in terms of the minimax $l_2$ risk. This has already been established for the horseshoe prior by \citet{PKV2014}, while the case for the other two families of prior distributions remain unanswered. In an insightful article, \citet{PS2011} argued that, in sparse high-dimensional problems, one should choose the prior distribution corresponding to the local shrinkage parameters to be appropriately heavy-tailed so that large signals can escape the ``gravitational pull'' of the global variance component and are almost left unshrunk which is essential for the recovery and identification of true signals.\newline 

This motivates us to consider in the first half of this article, the problem of estimating a sparse multivariate normal mean vector based on a very general class of tail robust one-group prior distributions. The aforesaid general class is rich enough to include a wide variety of one-group shrinkage priors, such as the three parameter beta normal mixtures, the generalized double Pareto priors, the inverse-gamma priors, the half-t priors and many more. We work under the same framework as in \citet{PKV2014} and assume that the number of non-zero entries of the true mean vector is known. Our goal here is to study asymptotic risk properties of Bayes estimates based on the general class of one-group priors under consideration and also to investigate the concentration properties of the corresponding posterior distributions in terms of the minimax optimal rate under the $l_2$ norm. For that we take the global shrinkage parameter as a tuning parameter that we are free to choose depending on the proportion of non-null means. It is shown that when the underlying true mean vector is sparse in the nearly-black sense, the mean square errors of Bayes estimates based on such tail robust priors, are within a constant factor of the corresponding minimax risk with respect to the $l_2-$norm, asymptotically. In particular, it is shown that for the ``horseshoe-type'' priors (to be defined in Section 2), such as, the three parameter beta normal mixtures with parameters $a=0.5$, $b>0$ (e.g. the horseshoe and the Strawderman$-$Berger priors), the generalized double Pareto priors with shape parameter $\alpha=1$ (e.g. the standard double Pareto prior) and the inverse-gamma priors with shape parameter $\alpha=0.5$, the corresponding Bayes estimates are asymptotically minimax in the sense that they asymptotically attain the minimax $l_2$ risk up to the correct constant, provided the global shrinkage parameter is asymptotically of the order of the proportion of non-null means or up to some logarithmic factor of it. This, to the best of our knowledge, is the first result in the literature which gives asymptotic minimaxity property up to the correct constant of Bayes estimates arising out of such one-group priors. Till now, such result is known only for the horseshoe prior, and that too, only up to a multiplicative constant. This provides a useful guideline while deciding over the choice of such one-group priors for the recovery of sparse signals.\newline

It is further shown that for a wide range of values of the global shrinkage parameter depending on the proportion of non-null means, the posterior distributions arising out of our general class of one-group priors under study, contract around the true mean vector at the minimax optimal rate under the $l_2$ norm and around the corresponding Bayes estimates at least as fast as this minimax rate for a broad range of choices of the global shrinkage parameter depending on the proportion of non-null means. Moreover, we also derive a lower found to the total posterior spread corresponding to an important sub-family of the ``horseshoe-type'' priors, which provides important insight regarding the choice of the global shrinkage parameter for this sub-family of priors. A major contribution of our theoretical investigation is to show that shrinkage priors which are appropriately heavy-tailed, are able to attain the minimax optimal rate of contraction and one does not need a sharp peak at the origin, provided the global tuning parameter is chosen carefully. We provide some novel unifying arguments that work for this general class under study and extend the work of \citet{PKV2014} over this class. As an immediate consequence of our general theoretical results, we settle the conjecture of \citet{BPPD2012} regarding such optimal posterior concentration properties of the normal$-$exponential$-$gamma and the generalized double Pareto priors.\newline

One key technique used for proving the aforesaid asymptotic minimaxity property turns out to be very handy in the study of asymptotic optimality (in a Bayesian decision theoretic framework) of simultaneous hypothesis testing for independent normal means using one-group shrinkage priors. The works of \citet{DG2013} and \citet{GTGC2015} already mentioned before established near asymptotic optimality properties of multiple testing rules based on one-group shrinkage priors when applied to data generated from a two-groups normal mixture model. Using an additive loss function, they showed that such multiple testing rules asymptotically attain the optimal Bayes risk for the corresponding two-groups formulation, up to a certain multiplicative factor. This indicates that the one-group solution can be a reasonable approximation asymptotically in the two-groups problem when the underlying mean vector is sparse. However, an interesting question that remains unanswered in these papers is whether such multiple testing procedures can asymptotically attain the optimal Bayes risk up to the correct constant. This question can be addressed satisfactorily by using the key technique mentioned before. We obtain exact asymptotic expressions of the Bayes risks of such multiple testing rules using some novel unifying arguments based on the aforesaid technique. It turns out that the answer is indeed in the affirmative for the ``horseshoe-type'' priors, whereby such decisions rules become asymptotically Bayes optimal under sparsity (ABOS), a notion that was introduced by \citet{BCFG2011} in the context of multiple testing. This finding is remarkable since it is the first result of its kind where the one-group answer achieves exactly the optimal performance asymptotically when applied to a two-groups formulation. Moreover, it is theoretically established that for the present multiple testing problem, when the global shrinkage parameter is treated as a tuning parameter only, its optimal choice corresponding to the horseshoe-type priors, should be asymptotically of the order of the proportion of true alternatives, provided this proportion is known. \newline

We organize the paper as follows. In Section 2, we describe the problem of estimating a sparse normal mean vector and the general class of tail robust shrinkage priors considered in this work. Section 3 contains the theoretical results involving asymptotic minimaxity of Bayes estimators arising out of this general class of shrinkage priors under study and the corresponding posterior contraction results. Section 4 contains the theoretical results on asymptotic Bayes optimality under sparsity of the induced multiple testing procedures under study, followed by some concluding remarks in Section 5. Proofs of the main theorems and other supporting results are given in the Appendix.

\subsection*{Notations and Definition}
We adopt the same convention of notation used in \citet{PKV2014} and \citet{GTGC2015}. Let $\{A_n\}$ and $\{B_n\}$ be any two sequences of non-negative real numbers indexed by $n$, such that, $B_n \neq 0$ for all sufficiently large $n$. We write $A_n\asymp B_n$ to denote $0 < \lim_{n\rightarrow\infty} \inf_{n} \frac{A_n}{B_n} \leqslant \lim_{n\rightarrow\infty} \sup_{n} \frac{A_n}{B_n} < \infty$ and $A_n \lesssim B_n $ to denote that there exists some constant $c > 0$ independent of $n$ such that $A_n \leqslant cB_n$, provided $n$ is sufficiently large. If $\lim_{n\rightarrow \infty} A_n/B_n=1$ we write it as $A_n \sim B_n$. Moreover, if $|\frac{A_n}{B_n}| \leq D $ for all sufficiently $n$ it is written as $A_n=O(B_n)$, where $D>0$ is some positive real number independent of $n$, and if $\lim_{n\rightarrow\infty}A_n/B_n=0$, it is expressed as $A_n=o(B_n)$. Thus $A_n=o(1)$ if $\lim_{n\rightarrow\infty}A_n=0$. Given any two non-negative real valued functions $f(x)$ and $g(x)$, both having a common domain of definition $(A,\infty)$, $A \geq 0$, such that $g(x)\neq 0$ for all sufficiently large $x$, we write $f(x) \sim g(x)$ as $x \rightarrow \infty$ to denote $\lim_{x \rightarrow \infty} f(x)/g(x)=1$.\newline

Throughout this paper, $Z$ is used to denote a $N(0,1)$ random variable having cumulative distribution function and probability density function $\Phi(\cdot)$ and $\phi(\cdot)$, respectively.

\begin{definition}\label{DEFN_SVF}
A positive measurable function $L$ defined over some $(A,\infty)$, $A \geqslant 0$, is said to be slowly varying (in Karamata's sense) if for each fixed $\alpha>0$, $L(\alpha x)\sim L(x)$ as $x\rightarrow\infty$.
\end{definition}

\section{Estimation of a Sparse Normal Mean Vector and a General Class of One-Group Tail Robust Priors}

Suppose that we observe an n-component random observation $(X_1,\dots,X_n) \in \mathbb{R}^n$, such that
\begin{equation}
 X_i = \theta_i + \epsilon_i  \mbox{  for } i =1,\dots,n, \label{NORMAL_MEANS_MODEL}
\end{equation}
where the unknown parameters $\theta_1,\dots,\theta_n$ denote the effects under investigation and $\epsilon=(\epsilon_1,\dots,\epsilon_n) \sim N_{n}(0, I_n)$.\newline

Let $l_0[q_n]$ denote the subset of $\mathbb{R}^n$ given by,
\begin{equation}
l_0[q_n] = \{\theta \in \mathbb{R}^n : \#(1 \leqslant j \leqslant n : \theta_j \neq 0) \leqslant q_n\}. 
\end{equation}
 
Suppose we want to estimate the unknown mean vector $\theta=(\theta_1,\dots,\theta_n)$. For that we assume $\theta$ to be sparse in the ``nearly black sense'', that is, $\theta \in l_0[q_n] $ with $q_n=o(n)$ as $n\rightarrow\infty$. Let $\theta_0=(\theta_{01},\dots,\theta_{0n})$ denote the true mean vector. Then the corresponding minimax rate with respect to the $l_2$-norm is given by (see \citet{DJHS1992}),
\begin{eqnarray}
 \inf_{\hat{\theta}}\sup_{\theta_{0}\in l_{0}[q_{n}]}E_{\theta_{0}}||\hat{\theta}-\theta_{0}||^2=2q_n\log\big(\frac{n}{q_n}\big)(1+o(1)),\mbox{ as } n\rightarrow\infty\label{MINIMAX_RATE}. 
 \end{eqnarray}
% $2q_n\log(n/q_n)(1+o(1))$ as $n\rightarrow\infty$ 
 In (\ref{MINIMAX_RATE}) and throughout this paper $E_{\theta_{0}}$ denotes an expectation with respect to the $N_{n}(\theta_0, I_n)$ distribution. Our goal is to find an estimate of $\theta$ from a Bayesian view point having some good theoretical properties. As stated already in the introduction that a natural Bayesian approach to model (\ref{NORMAL_MEANS_MODEL}) is to use a two-component point mass mixture prior for the $\theta_i$'s, given by,
\begin{equation}\label{TWO_GROUP_THETA}
 \theta_i \stackrel{i. i. d.}{\sim} (1-p)\delta_{\{0\}} + p\cdot f, \mbox{ $i=1,\dots,n.$}
 \end{equation} 
where $\delta_{\{0\}}$ denotes the distribution having probability mass 1 at the point 0, and $f$ denotes an absolutely continuous distribution over $\mathbb{R}$. Here the mixing proportion $p$ is often interpreted as the theoretical proportion of non-null $\theta_i$'s. See \citet{MB1988} and \citet{JS2004} in this context. It is usually recommended to choose a heavy tailed absolutely continuous distribution $f$ over $\mathbb{R}$ so that large observations can be recovered with higher degree of accuracy. \citet{JS2004} used a $t$ distribution in this context and used an empirical Bayes approach in order to estimate the unknown mixing proportion $p$ via the method of marginal maximum likelihood and showed that if the coordinate-wise posterior median estimate is used, the resulting estimator attains the minimax rate with respect to the $l_r$ loss, $r \in (0, 2]$. \citet{CV2012} studied the full Bayes approach where they found conditions on the two-groups prior that ensure contraction of the posterior distribution at the minimax rate. A comprehensive list of other empirical Bayes approaches using a two-groups model can be found in \citet{YL2005}, \citet{JZ2009}, \citet{CV2012},  \citet{MW2014} and references therein.\newline

 As already mentioned in the introduction that although the two groups prior (\ref{TWO_GROUP_THETA}) is considered to be the most natural formulation for handling sparsity from a Bayesian view point, it offers a daunting computational challenge in high-dimensional problems. Due to this reason, the one-group formulation to model sparse data has received considerable attention from researchers over the years. \citet{PS2011} showed that almost all such shrinkage priors can be expressed as multivariate scale-mixture of normals. In this article, we consider Bayes estimators based on a general class of one-group shrinkage priors given through the following hierarchical one-group formulation:
  \begin{align*}
  \begin{array}{ll}\label{ONE_GRP_PRIOR}
  X_i|\theta_i &\sim \mbox{ } N(\theta_i,1), \mbox{ independently for } i=1,\dots,n\\
  \theta_{i}|(\lambda_i^2, \tau^2) &\sim \mbox{ }N(0,\lambda_i^2\tau^2), \mbox{ independently for } i=1,\dots,n \nonumber\\
  \lambda_i^2 &\sim \mbox{ } \pi(\lambda_i^2), \mbox{ independently for } i=1,\dots,n
  \end{array}
 \end{align*} 
 with $\pi(\lambda_i^2)$ being given by,
 \begin{equation}
 \label{LAMBDA_PRIOR}
  \pi(\lambda_i^2) =K (\lambda_i^2)^{-a-1}L(\lambda_i^2),
 \end{equation}
 where $K \in (0,\infty)$ is the constant of proportionality, $a$ is a positive real number and $L:(0,\infty) \rightarrow (0,\infty)$ is a measurable, non-constant slowly varying function. For the theoretical development in this paper, we assume that the function $L(\cdot)$ in (\ref{LAMBDA_PRIOR}) satisfies the following:
\begin{assumption}\label{ASSUMPTION_LAMBDA_PRIOR}
\qquad
 \begin{enumerate}
 \item  $\lim_{t \rightarrow\infty}L(t) \in (0,\infty),$ that is, there some exists $c_{0}(>0)$ such that $L(t) \geqslant c_{0}$ for all $t\geqslant t_{0}$, for some $t_{0}>0$, which depends on both $L$ and $c_{0}$. 
 
 \item There exists some $0 < M < \infty $ such that $\sup_{t\in(0,\infty)}L(t) \leqslant M.$
\end{enumerate}
\end{assumption}

Each $\lambda_i^2$ is referred to as a local shrinkage parameter and the parameter $\tau$ is called the global shrinkage parameter. For estimation of the unknown mean vector in the present normal means model (\ref{NORMAL_MEANS_MODEL}), we treat $\tau$ as a tuning parameter that we are free to choose depending on the proportion of non-null means. From Theorem 1 of \citet{PS2011} it follows that the above general class of one-group priors will be ``tail-robust'' in the sense that for any given $\tau>0$, $E(\theta_i|X_i,\tau) \approx X_i,$ for large $X_i$'s. We would like to mention here that a very broad class of one-group shrinkage priors actually fall inside this general class. \citet{GTGC2015} established that the three parameter beta normal mixtures and the generalized double Pareto priors can be expressed in the above general form by showing that the corresponding prior distribution of the local shrinkage parameters can be written in the form given in (\ref{LAMBDA_PRIOR}) with the corresponding $L(\cdot)$ satisfying Assumption \ref{ASSUMPTION_LAMBDA_PRIOR}. It is easy to verify that some other well known shrinkage priors such as the families of inverse-gamma priors and the half-t priors are also covered by this general class of prior distributions under consideration. In case $a=0.5$ for the general class of shrinkage priors under study, we refer the corresponding one-group priors as the {\it horseshoe-type} priors. The class of horseshoe-type priors encompasses an important sub-family of the general class of one-group priors under study, such as, the three parameter beta normal mixtures with parameters $a=0.5$, $b>0$ (e.g. the horseshoe, the Strawderman$-$Berger), the generalized double Pareto priors with shape parameter $\alpha=1$ (e.g. the standard double Pareto), the inverse-gamma prior with shape parameter $\alpha=0.5$ and many other shrinkage priors.\newline

% From Theorem 1 of \citet{PS2011} it follows that the above general class of one-group priors will be ``tail-robust'' in the sense that for any given $\tau > 0$, $E(\theta_i|X_i,\tau) \approx X_i,$ for large $X_i$'s. \citet{PS2011} argued that in a sparse problem, the global shrinkage parameter $\tau$ should be very small so that small $X_i$'s or the noise observations can be shrunk towards the origin while the prior distribution of the local shrinkage parameters $\lambda_i^2$ should have heavy tails so that large signals can escape the effect of $\tau$ and almost remain unshrunk. Thus (\ref{LAMBDA_PRIOR}) should result in a prior distribution for the $\theta_i$'s which has a high concentration of mass near the origin but have thick tails at the extremes to accommodate large signals. \citet{PS2011} also showed that for priors having exponential or lighter tails, such as the Laplace or the double$-$exponential prior, even the large $X_i$'s will always be shrunk towards the origin by some non-diminishing amount for small values of $\tau$.

Now for a general global-local scale mixtures of normals, we have for each $i=1,\dots,n$,
% \begin{equation}
%  \theta_{i}|(X_i, \lambda_i,\tau) \sim N((1-\kappa_i)X_i,(1-\kappa_i)), \mbox{ } \kappa_i=1/(1+\lambda_i^2\tau^2),\nonumber
% \end{equation}
% independently for $i=1,\dots,n$, so that for each $i$, the posterior mean of $\theta_i$ is given by,
% \begin{equation} \label{POST_MEAN_OG1}
%  E(\theta_{i}|X_i,\lambda_i,\tau)=(1-\kappa_i)X_i, .
% \end{equation} 
% Next, using the iterated expectation formula it follows that,
   \begin{equation} \label{POST_MEAN_OG2}
   E(\theta_{i}|X_i,\tau)=(1-E(\kappa_i|X_i,\tau))X_i,
   \end{equation}
where $\kappa_i=1/(1+\lambda_i^2\tau^2)$ is called the $i-$th posterior shrinkage coefficient. The resulting posterior mean $E(\theta|X,\tau)=(E(\theta_{1}|X_1,\tau),\dots,E(\theta_{n}|X_n,\tau))$ will be the Bayes estimate used to recover $\theta$ and will be denoted by $T_{\tau}(X)$. For notational convenience, we shall denote $E(\theta_{i}|X_i,\tau)$ by $T_{\tau}(X_i)$. It will be shown in the next section that when the true mean vector $\theta_0$ is sparse in the ``nearly black sense'' and $a \in [0.5,1)$, the estimator $T_{\tau}(X)$ will asymptotically attain the minimax rate (\ref{MINIMAX_RATE}) with respect to the $l_2$ norm, up to some multiplicative constant and that the posterior distribution contracts around the true $\theta_0$ at this minimax optimal rate for suitably chosen values of $\tau$ depending on the proportion $q_n/n$. In particular, it will be shown that for the horseshoe-type priors, the resulting Bayes estimates will be asymptotically minimax in the sense that the ratio of the corresponding mean square error to the minimax risk in (\ref{MINIMAX_RATE}) asymptotically becomes 1 as the dimension $n$ grows to infinity.

\section{Asymptotic Minimaxity and Posterior Contraction Rates}
In this section, we present the theoretical results involving the mean square error for the Bayes estimates arising out of the general class of tail robust shrinkage priors under study, with $a\in [0.5,1)$, and the spread of the corresponding posterior distributions. It is assumed that the number of non-zero components $q_n$ of the unknown mean vector is known. Theorem \ref{THM_3.1} gives an upper bound on the mean square error for the Bayes estimates arising out of the general class one-group priors under consideration. Using this upper bound, it follows that for various choices of the global shrinkage parameter $\tau$, depending on the proportion of non-zero means $\frac{q_n}{n}$, the aforesaid Bayes estimates attain the minimax risk with respect to the $l_2$-norm up to some multiplicative constants. In particular, for the horseshoe-type priors, the corresponding Bayes estimates are shown to be asymptotically minimax. Theorem \ref{THM_3.2} gives an upper bound to the total posterior spread of our chosen one-group priors. Theorem \ref{THM_3.3} provides a sharp upper bound to the rate of contraction around the true mean vector for the posterior distributions arising out of the general class of shrinkage priors under study, which shows that such posterior distributions contract around the true mean vector at the minimax $l_2$ rate. Moreover, it also shows that these posterior distributions contract around the corresponding Bayes estimates at least as fast as the minimax optimal rate in the $l_2$ norm. Theorem \ref{THM_3.4} provides a lower bound to the total posterior variance for an important subclass of the horseshoe-type priors that gives more insight about the spread of such posterior distributions around the corresponding Bayes estimators for various choices of $\tau$. Proofs of most of these results are given in the Appendix. Although we build on certain ideas of the proofs of the main theorems of \cite{PKV2014}, we have to come up with novel unifying technical arguments that work for the kind of one-group priors studied in this paper. In particular, Lemmas \ref{LEMMA_MOST_IMP} - \ref{LEMMA_IMPORTANT_INTEGRAL}, given in the Appendix, which form the crux of the arguments for proving the main theoretical results, namely, Theorems \ref{THM_3.1} - \ref{THM_3.3}, are completely independent of the work of \citet{PKV2014}. However, proof of Theorem \ref{THM_3.4} follows using some key arguments of \citet{PKV2014}. Our work shows that some of the technical arguments used in \citet{PKV2014} can be used in greater generality.

% Lemma A.1 and Theorem 3.4 in
\begin{thm}\label{THM_3.1}
Suppose $X \sim \mathcal{N}_{n}(\theta_{0},I_{n})$, where $\theta_{0} \in l_{0}[q_{n}]$. Consider the general class of shrinkage priors where the prior distribution of the local shrinkage parameters $\lambda_i^2$'s is given by (\ref{LAMBDA_PRIOR}) with $0.5 \leqslant a < 1$ and the corresponding slowly varying component $L(\cdot)$ satisfies Assumption \ref{ASSUMPTION_LAMBDA_PRIOR}. Then the corresponding Bayes estimate $T_{\tau}(X)$, based on this general class of shrinkage priors, satisfies
 \begin{equation}\label{MSE_BAYES_ESTIMATE}
 \sup_{\theta_{0}\in l_{0}[q_{n}]}E_{\theta_{0}}||T_{\tau}(X)-\theta_{0}||^2
 \lesssim q_n\log\big(\frac{1}{\tau^{2a}}\big)+(n-q_n)\tau^{2a}\sqrt{\log\big(\frac{1}{\tau^{2a}}\big)}
 \end{equation}
 if $\tau \rightarrow 0$, as $n\rightarrow\infty$, $q_n\rightarrow\infty$ and $q_n=o(n)$.
\end{thm}
\begin{proof}
 See Appendix.
\end{proof}

Using Theorem \ref{THM_3.1}, it immediately follows that Bayes estimators arising out of the general class of shrinkage priors under study with $a\in[0.5,1)$, attain the minimax risk with respect to the $l_2-$norm, possibly up to some multiplicative factors, for various choices of the global shrinkage parameter $\tau$. To see this, let us first fix any constant $c>1$ and choose any $\rho >c$ in the proof of Theorem \ref{THM_3.1} (see Appendix). Then the corresponding multiplicative factor before the upper bound in (\ref{MSE_BAYES_ESTIMATE}) can at most be $4a\rho^2$. Now taking $\tau=(q_n/n)^{\alpha}$, $\alpha \geqslant 1$, or $\tau=(q_n/n)\sqrt{\log(n/q_n)}$ in (\ref{MSE_BAYES_ESTIMATE}), it follows that the corresponding mean square error can at most be of the order of $2q_n\log\big(\frac{n}{q_n}\big)$ up to the multiplicative factor $2a\rho^2\max\{1,\alpha\}$, while it is always bounded below by the minimax $l_2-$risk which is of the order $2q_n\log\big(\frac{n}{q_n}\big)$. Clearly,
 \begin{equation}\label{BAYES_ESTIMATE_MINIMAXITY}
 \sup_{\theta_{0}\in l_{0}[q_{n}]}E_{\theta_{0}}||T_{\tau}(X)-\theta_{0}||^2
 \asymp q_n\log\big(\frac{n}{q_n}\big).
 \end{equation}

 However, a more interesting and remarkable consequence of Theorem \ref{THM_3.1} is the following. Consider the family of horseshoe-type priors for which one has $a=0.5$ and let us take $\tau=q_n/n$ or $\tau=(q_n/n) \sqrt{\log(n/q_n)}$ in (\ref{MSE_BAYES_ESTIMATE}). Then using the proof of Theorem \ref{THM_3.1}, we have,
 \begin{equation}\label{MSE_HSTYPE_UB_RHO}
 \sup_{\theta_{0}\in l_{0}[q_{n}]}E_{\theta_{0}}||T_{\tau}(X)-\theta_{0}||^2
 \leqslant 2\rho^2q_n\log\big(\frac{n}{q_n}\big)(1+o(1))\mbox{ as } n\rightarrow\infty,
 \end{equation}
 where the $o(1)$ term depends on $\rho>c>1$. Now, using (\ref{MSE_HSTYPE_UB_RHO}), and the fact that the minimax $l_2$ risk in (\ref{MINIMAX_RATE}) is the greatest lower bound to the mean square error term in (\ref{MSE_HSTYPE_UB_RHO}), we obtain
 \begin{equation}\label{RATIO_MSE_HSTYPE_MINIMAX_UB_RHO}
 1\leqslant\frac{\sup\limits_{\theta_{0}\in l_{0}[q_{n}]}E_{\theta_{0}}||T_{\tau}(X)-\theta_{0}||^2}{\inf\limits_{\hat{\theta}}\sup\limits_{\theta_{0}\in l_{0}[q_{n}]}E_{\theta_{0}}||\hat{\theta}-\theta_{0}||^2}\leqslant\rho^2(1+o(1)) \mbox{ as } n\rightarrow\infty.
 \end{equation}
%  For simplicity of notation let $R_{OG,n}(l_{0}[q_{n}])$ and $R_{Opt,n}(l_{0}[q_{n}])$ respectively denote the mean square error 
 Hence taking limit inferior and limit superior in (\ref{RATIO_MSE_HSTYPE_MINIMAX_UB_RHO}) as $n\rightarrow\infty$, we have,
  \begin{equation}\label{RATIO_MSE_HSTYPE_MINIMAX_LIMINF_LIMSUP}
 1\leqslant \liminf_{n\rightarrow\infty}\frac{\sup\limits_{\theta_{0}\in l_{0}[q_{n}]}E_{\theta_{0}}||T_{\tau}(X)-\theta_{0}||^2}{\inf\limits_{\hat{\theta}}\sup\limits_{\theta_{0}\in l_{0}[q_{n}]}E_{\theta_{0}}||\hat{\theta}-\theta_{0}||^2}\leqslant\limsup_{n\rightarrow\infty}\frac{\sup\limits_{\theta_{0}\in l_{0}[q_{n}]}E_{\theta_{0}}||T_{\tau}(X)-\theta_{0}||^2}{\inf\limits_{\hat{\theta}}\sup\limits_{\theta_{0}\in l_{0}[q_{n}]}E_{\theta_{0}}||\hat{\theta}-\theta_{0}||^2}\leqslant\rho^2.
 \end{equation}
 Note that the Bayes estimator $T_{\tau}(X)$ does not depend on the choice of $\rho>c>1$. Hence the ratio in (\ref{RATIO_MSE_HSTYPE_MINIMAX_UB_RHO}) is independent of the choice of $\rho>c>1$. Consequently, the limit inferior and limit superior terms in (\ref{RATIO_MSE_HSTYPE_MINIMAX_LIMINF_LIMSUP}) are both independent of how $\rho>c>1$ are chosen. But choices of $\rho>c>1$ in (\ref{RATIO_MSE_HSTYPE_MINIMAX_LIMINF_LIMSUP}) in are arbitrary. Therefore, taking infimum over all possible choices of $\rho>c>1$ in (\ref{RATIO_MSE_HSTYPE_MINIMAX_LIMINF_LIMSUP}) it follows that
%    \begin{equation}\label{}
%  1\leqslant \liminf_{n\rightarrow\infty}\frac{\sup\limits_{\theta_{0}\in l_{0}[q_{n}]}E_{\theta_{0}}||T_{\tau}(X)-\theta_{0}||^2}{\inf\limits_{\hat{\theta}}\sup\limits_{\theta_{0}\in l_{0}[q_{n}]}E_{\theta_{0}}||\hat{\theta}-\theta_{0}||^2}\leqslant\limsup_{n\rightarrow\infty}\frac{\sup\limits_{\theta_{0}\in l_{0}[q_{n}]}E_{\theta_{0}}||T_{\tau}(X)-\theta_{0}||^2}{\inf\limits_{\hat{\theta}}\sup\limits_{\theta_{0}\in l_{0}[q_{n}]}E_{\theta_{0}}||\hat{\theta}-\theta_{0}||^2}\leqslant1,
%  \end{equation}
%  whence we have,
 \begin{equation}\label{EXACT_ASYMP_MINIMAXITY}
  \sup\limits_{\theta_{0}\in l_{0}[q_{n}]}E_{\theta_{0}}||T_{\tau}(X)-\theta_{0}||^2 \sim \inf\limits_{\hat{\theta}}\sup\limits_{\theta_{0}\in l_{0}[q_{n}]}E_{\theta_{0}}||\hat{\theta}-\theta_{0}||^2.
 \end{equation}

Observe that the above arguments also go through even if $\tau$ is taken to be asymptotically of the order of $\frac{q_n}{n}$ or $\frac{q_n}{n}\sqrt{\log(n/q_n)}$. Thus, (\ref{EXACT_ASYMP_MINIMAXITY}) clearly shows that for carefully chosen values of the global tuning parameter $\tau$ depending on the proportion of non-null means $\frac{q_n}{n}$, Bayes estimators based on the horseshoe-type priors, asymptotically attains the minimax risk with respect to the $l_2-$norm, up to the correct constant, and are therefore asymptotically minimax. It may be noted that it is a refinement of the corresponding result on asymptotic minimaxity (up to a multiplicative constant) of the horseshoe estimator obtained by \citet{PKV2014}. One possible explanation for such good performance of these priors is their ability to squelch the noise observations back to the origin, while leaving the large observations almost unshrunk, provided the global shrinkage parameter $\tau$ is carefully chosen. It is well known that smaller values of $a$ typically result in a prior distribution with heavier tails. For example, the inverted beta families with $a=0.5$ typically yields Cauchy like tails. See \citet{PS2012} in this context. Similar discussion on how the choice of the parameter $a$ controls the tail behavior of the generalized double Pareto priors can also be found in \citet{ADL2012}. Recall that for the generalized double Pareto priors, we have $a=\alpha/2$, where $\alpha$ denotes the corresponding shape parameter (see \cite{GTGC2015}). \citet{ADL2012} argued for using smaller values of $\alpha$ to ensure strong shrinkage of noise-like observations towards the origin and to avoid the occurrence of large bias terms due to large signals. They recommended the standard double Pareto distribution as a default prior specification which has Cauchy like tail and for which one has $a=0.5$. It is also worth noting in this context that \citet{ADC2011} recommended using $a \in (0,1)$ and $b \in (0,1)$ for their proposed three parameter beta normal mixture priors. Their range of $b$ is fully covered, while that of $a$ is partially covered by Theorem \ref{THM_3.1}. Thus, Theorem \ref{THM_3.1} provides strong theoretical support in favor of the recommendations of \citet{ADC2011} and \citet{ADL2012} regarding the choice of the hyperparameters in the corresponding one-group formulation.

 \begin{remark}
  Note that the aforesaid asymptotic minimaxity property of Bayes estimates based on our chosen class of priors depends on treating $\tau$ as a tuning parameter to be chosen carefully depending on the proportion of non-zero means. \cite{ADC2011} argued that the choice of the global tuning parameter should reflect the prior knowledge about the underlying sparsity presented in the data, provided such information is available. However, in practice one often doesn't have such information. In such situations, \citet{PKV2014} provided conditions under which the horseshoe estimator combined with an empirical Bayes estimate of the global variance component still attains the minimax risk when $q_n \propto n^{\beta}$, for $0<\beta<1$. We record here that we have a partly independent argument for proving such asymptotic minimaxity result based on the empirical Bayes estimate proposed by \citet{PKV2014}. This argument works for the general class of priors under study in the current paper. However, considering the length of the article, we omit the proof.
 \end{remark}
The next theorem gives an upper bound to the total posterior variance corresponding to our general class of heavy tailed shrinkage priors when $a\in[0.5,1)$.

\begin{thm}\label{THM_3.2}
 Suppose $X \sim \mathcal{N}_{n}(\theta_{0},I_{n})$, where $\theta_{0} \in l_{0}[q_{n}]$. Consider the general class of shrinkage priors where the prior distribution of the local shrinkage parameters $\lambda_i^2$'s is given by (\ref{LAMBDA_PRIOR}) with $0.5\leqslant a <1$ and the corresponding slowly varying component $L(\cdot)$ satisfies Assumption \ref{ASSUMPTION_LAMBDA_PRIOR}. Then the total posterior variance corresponding to this general class of shrinkage priors satisfies
 \begin{equation}\label{EXPECT_POST_VAR_UB}
 \sup_{\theta_{0}\in l_{0}[q_{n}]}E_{\theta_{0}} \sum_{i=1}^{n} Var(\theta_{0i}|X_i) \lesssim
  q_n\log\big(\frac{1}{\tau^{2a}}\big)+(n-q_n)\tau^{2a}\sqrt{\log\big(\frac{1}{\tau^{2a}}\big)}.
%  \mathop{\sup_{\theta_{0}\in l_{0}[q_{n}]}E_{\theta_{0}} \sum_{i=1}^{n} Var(\theta_{0i}|X_i) \lesssim}
%  \left\{
%  \begin{array}{ll}
%   q_n+(n-q_n)\tau^{2a}\sqrt{\log\big(\frac{1}{\tau^{2a}}\big)}, & \mbox{if}\; a=\frac{1}{2} ,\\ \\
%   q_n\log\big(\frac{1}{\tau^{2a}}\big)+(n-q_n)\tau^{2a}\sqrt{\log\big(\frac{1}{\tau^{2a}}\big)}, & \mbox{if}\; \frac{1}{2}<a<1,
%  \end{array}\right.
 \end{equation}
 
if $\tau \rightarrow 0$, as $n\rightarrow\infty$, $q_n\rightarrow\infty$ and $q_n=o(n)$.
\end{thm}
\begin{proof}
 See Appendix.
\end{proof}

The next theorem gives upper bounds on the rate of contraction of posterior distributions based on our general class of one-group priors under study with $0.5\leqslant a <1$, both around the true mean vector as well as the corresponding Bayes estimates. 

\begin{thm}\label{THM_3.3}
 Under the assumptions of Theorem \ref{THM_3.2}, if $\tau=\big(\frac{q_n}{n}\big)^{\alpha}$, $\alpha\geqslant1$, or $\tau=\frac{q_n}{n}\sqrt{\log(n/q_n)}$, then
 \begin{equation}\label{OPT_POST_CON_ARND_TRUE_MEAN}
\sup_{\theta_{0} \in l_{0}[q_{n}]} E_{\theta_{0}} \Pi \bigg(\theta:||\theta-\theta_0||^2 > M_{n}q_n\log\big(\frac{n}{q_n}\big)|X\bigg) \rightarrow 0,
 \end{equation}
and
 \begin{equation}\label{POST_CON_ARND_BAYES_EST}
\sup_{\theta_{0} \in l_{0}[q_{n}]} E_{\theta_{0}} \Pi \bigg(\theta:||\theta-T_{\tau}(X)||^2 > M_{n}q_n\log\big(\frac{n}{q_n}\big)|X\bigg) \rightarrow 0,
 \end{equation}
for every $M_n\rightarrow\infty$ as $n\rightarrow\infty$.
\end{thm}

\begin{proof}
 A straightforward application of Markov's inequality coupled with the results of Theorem \ref{THM_3.1} and Theorem \ref{THM_3.2} leads to (\ref{OPT_POST_CON_ARND_TRUE_MEAN}), while (\ref{POST_CON_ARND_BAYES_EST}) follows from the result of Theorem \ref{THM_3.2} together with the Markov's inequality.
\end{proof}
Using (\ref{OPT_POST_CON_ARND_TRUE_MEAN}) and (\ref{POST_CON_ARND_BAYES_EST}), it follows that the posterior distributions based on our chosen class of one-group priors, with $0.5\leqslant a <1$, contract around both the true mean vector and the corresponding Bayes estimates at least as fast as the minimax $l_2$ risk in (\ref{MINIMAX_RATE}). On the other hand, \citet{GGV2000} showed that the posterior distributions cannot contract faster than the minimax risk around the truth. Hence, the rate of contraction of these posterior distributions around the true $\theta_{0}$ must be the minimax optimal rate in (\ref{MINIMAX_RATE}), up to some multiplicative constants. However, the same may not be true for contraction around the corresponding Bayes estimates since no such result is available in the literature on the rates of contraction of posterior distributions around their respective posterior means. As already mentioned in the introduction that in order to provide a realistic measure of uncertainty, a posterior distribution needs to contract around the corresponding posterior mean at the same rate at which it approaches towards the true parameter value. Hence one needs to show that the posterior distributions in Theorem \ref{THM_3.3} contract around their corresponding Bayes estimates at the minimax optimal rate which cannot be established through (\ref{POST_CON_ARND_BAYES_EST}). However, it seems that this requires a deeper investigation by invoking newer techniques and arguments to come up with a satisfactory answer to this question. We leave this as an important and interesting problem for future research.\newline

In order to obtain a better insight regarding the spread of the posterior distribution around the Bayes estimates and the effect of choosing different values of $\tau$ on them, let us confine our attention to an important sub-family of the class of horseshoe-type prior distributions when the corresponding $L(\cdot)$ in (\ref{LAMBDA_PRIOR}) satisfies Assumption \ref{ASSUMPTION_LAMBDA_PRIOR} and is non-decreasing over $(0,\infty)$. This sub-family of priors covers the three parameter beta normal mixtures with hyperparameters $a=0.5$, $b>0$ (e.g. the horseshoe, the Strawderman$-$Berger), the generalized double Pareto priors with shape parameter $\alpha=1$ (e.g. the standard double Pareto), the inverse-gamma priors with shape parameter $\alpha=0.5$ and many other such shrinkage priors. The next theorem gives a lower bound to the total posterior variance corresponding to this sub-family of priors which provides some useful insights into the effect of choosing $\tau$ depending on the proportion of non-zero means $\frac{q_n}{n}$. 

\begin{thm}\label{THM_3.4}
 Suppose $X \sim \mathcal{N}_{n}(\theta_{0},I_{n})$, where $\theta_{0} \in l_{0}[q_{n}]$. Consider the general class of shrinkage priors where the prior distribution of the local shrinkage parameters $\lambda_i^2$'s is given by (\ref{LAMBDA_PRIOR}) with $a=0.5$, and the corresponding slowly varying component $L(\cdot)$ satisfies Assumption \ref{ASSUMPTION_LAMBDA_PRIOR} and is non-decreasing over $(0,\infty)$. Then, the total posterior variance corresponding to the general class of shrinkage priors, satisfies
 \begin{equation}\label{LOWER_BND_POST_VAR}
  \sum_{i=1}^{n} E_{\theta_{0}} Var(\theta_{0i}|X_i) \gtrsim (n-q_n)\tau\sqrt{\log\big(\frac{1}{\tau}\big)},
\end{equation}
 if $\tau \rightarrow 0$, as $n\rightarrow\infty$, $q_n\rightarrow\infty$ and $q_n=o(n)$.
\end{thm}
\begin{proof}
 See Appendix.
\end{proof}

Like \citet{PKV2014}, here also it can be observed from Theorem \ref{THM_3.4} that if we take $\tau=(q_n/n)^{\alpha}$ for $0<\alpha<1$, then the lower bound in (\ref{LOWER_BND_POST_VAR}) may exceed the minimax rate which implies that for such values of $\tau$ the corresponding prior distributions may have a sub-optimal posterior spread. Again, taking $\tau=q_n/n$ in Theorem \ref{THM_3.4}, the corresponding lower bound in (\ref{LOWER_BND_POST_VAR}) is of the order $(q_n/n)\sqrt{\log(q_n/n)}$, thereby missing the minimax rate by the logarithmic factor $\sqrt{\log(q_n/n)}$. This also suggests that for $\tau=q_n/n$ such posterior distributions may contract around their respective centers at a rate faster than the minimax rate. Again, taking $\tau=(q_n/n)^{\alpha}$ with $\alpha \geqslant 1$ in the lower bound in Theorem \ref{THM_3.4} results a rate that is faster than the minimax $l_2$ risk, which means that for such choices of $\tau$ the corresponding lower bound in (\ref{LOWER_BND_POST_VAR}) fail to provide adequate information. However, if instead we choose $\tau=(\frac{q_n}{n})\sqrt{\log(n/q_n)}$ in Theorem \ref{THM_3.4}, then using Theorem \ref{THM_3.2} and Theorem \ref{THM_3.4} it follows that the total posterior spread for this sub-family of priors is asymptotically of the order of the minimax $l_2$ risk in (\ref{MINIMAX_RATE}). Again, it has already been established that all the desired upper bounds in Theorem \ref{THM_3.1} and Theorem \ref{THM_3.2} are asymptotically of the order of the minimax error rate under the $l_2$ norm. Thus, like \citet{PKV2014}, our theoretical results also suggest that for estimation of a sparse multivariate normal mean vector based on the horseshoe-type priors such as those in Theorem \ref{THM_3.4} and also for optimal posterior contraction properties, $\tau=(\frac{q_n}{n})\sqrt{\log(n/q_n)}$ should be regarded as the optimal choice of $\tau$.

\section{Asymptotic Bayes Optimality Under Sparsity}
In the second part of this paper, we focus on the problem of simultaneous testing for the means of independent normal observations. For that we consider the normal means model in (\ref{NORMAL_MEANS_MODEL}), where we have $n$ independent observations $X_1,\dots,X_n$ such that for each $i$, $X_i$ is distributed according to a $N(\theta_i,1)$ distribution. Suppose we wish to know whether for each $i$, $\theta_i$ is zero or not, that is, we wish to test $H_{0i}:\theta_i=0$ against $H_{Ai}:\theta_i\neq 0$, for $i=1,\dots,n$. Our focus is on situations when the unknown mean vector $(\theta_1,\dots,\theta_n)$ is sparse, that is, most of the null hypotheses $H_{0i}$'s are assumed to be true compared to the total number of tests $n$, which is typically assumed to be large. As already mentioned in the introduction that a natural Bayesian approach to formulate problems of this kind is to model the data through a two-component point mass mixture prior for the unknown $\theta_i$'s. Let us introduce a set of latent indicator random variables $\nu_1,\dots,\nu_n$, where $\nu_i=0$ denotes the event that $H_{0i}$ is true while $\nu_i=1$ corresponds to the event $H_{0i}$ is false. Note that under $H_{0i}$, $\theta_i=0$, while under $H_{Ai}$ $\theta_i\neq 0$. So, let us assume that given $\nu_i=0$, $\theta_i\sim\delta_{\{0\}}$, the distribution having probability mass $1$ at the point $0$, while $\theta_i|\nu_i=1\sim N(0,\psi^2)$, where $\psi^2 > 0$ is usually assumed to be large to accommodate the large signals. It is further assumed that the unobservable indicator random variables $\nu_i$'s are random samples from a $\mbox{Bernoulli}(p)$ distribution, for some $p\equiv p_n$ in $(0,1)$. The parameter $p$ is often interpreted as the theoretical proportion of true alternatives. Marginalizing over the $\nu_i$'s, given $(p,\psi^2)$, $\theta_i$'s are assumed to be generated according to the following two-component point mass mixture distribution given by, 
 \begin{equation}\label{TWO_GROUP_MU}
  \theta_i \stackrel{i. i. d.}{\sim} (1-p)\delta_{\{0\}} + p N(0,\psi^2), \mbox{ $i=1,\dots,n$}.
   \end{equation} 
The marginal distribution of $X_i$'s is given by the following {\it two-groups} normal mixture model:
 \begin{equation}\label{TWO_GROUP_X}
  X_i \stackrel{i. i. d.}{\sim} (1-p)N(0,1) + p N(0,1+\psi^2), \mbox{ $i=1,\dots,n$}.
   \end{equation} 
Under the above set up, the given testing problem now boils down to testing simultaneously
\begin{equation}\label{TO_TEST}
 H_{0i}: \nu_i=0  \mbox{ versus } H_{Ai}: \nu_i=1 \mbox{ for } i=1,\dots,n.
\end{equation}
 
For the above multiple testing problem in (\ref{TO_TEST}), we consider a Bayesian decision theoretic framework described as follows. Let us assume that for each individual testing problem, the loss for committing a type I error and a type II error are the same and equal to $1$ and the total loss is assumed to be the sum of losses incurred in each individual test. Thus the overall loss for the present multiple testing problem is the total number of misclassified hypotheses. Suppose $t_{1i}$ and $t_{2i}$ denote the probabilities of committing a type I error and a type II error respectively of a given multiple testing procedure for the $i-$th testing problem. Then, the corresponding Bayes risk of that procedure under the two-groups model (\ref{TWO_GROUP_X}), denoted $R$, is given by   
 \begin{equation}\label{BAYES_RISK_GEN}
  R=\sum \limits_{i=1}^{n}\big[(1-p)t_{1i}+pt_{2i}\big].
 \end{equation}
Under this set up, \cite{BCFG2011} showed that the Bayes rule which minimizes the Bayes risk in (\ref{BAYES_RISK_GEN}) is the test which, for each $i=1,\dots,n$, declares the $i$-th null hypothesis $H_{0i}$ to be significant if
 \begin{equation}\label{BAYES_RULE}
 \pi(\nu_i=1|X_i)> 0.5, \mbox{ or equivalently, } X_i^2 > c^2,
 \end{equation}
 where 
 \begin{equation}\label{THRESHOLD}
  c^2\equiv c^2_{\psi,f,\delta}=\frac{1+\psi^2}{\psi^2}(\log(1+\psi^2)+ 2\log(\frac{1-p}{p})).\nonumber
 \end{equation}

 The above rule is also referred to as the Bayes Oracle by \cite{BGT2008} and \cite{BCFG2011}, since it involves the unknown mixing proportion $p$ and the variance $\psi^2$ of the non-null $\theta_i$'s which is also unknown, and hence cannot be attained for finite $n$.\newline
 
 By introducing two new parameters $u\equiv u_n=\psi_n^2 $ and $v\equiv v_n=\psi_n^2\big(\frac{1-p_n}{p_n}\big)^2$,  \cite{BCFG2011} considered the following asymptotic scheme given by,

 \begin{assumption}\label{ASSUMPTION_ASYMP}
 \qquad
 \begin{enumerate}

    \item $p_n \rightarrow 0 \mbox{ as }n\rightarrow\infty$.
    
    \item $u_n=\psi_n^2 \rightarrow\infty \mbox{ as }n\rightarrow\infty$.
    
    \item $\frac{\log{v_n}}{u_n} \rightarrow C \in (0,\infty) \mbox{ as }n\rightarrow\infty$.
  \end{enumerate}
 \end{assumption}
 
 Under Assumption \ref{ASSUMPTION_ASYMP}, \cite{BCFG2011} obtained the following asymptotic expressions for the probabilities of type I and type II errors corresponding to the Bayes Oracle (\ref{BAYES_RULE}), given by,
   \begin{eqnarray}
    t_1^{BO} &=& e^{-C/2} \sqrt{\frac{2}{\pi v \log v}}(1+o(1)), \mbox{ and }\label{T1_OPT}\\
    t_2^{BO} &=& (2\Phi(\sqrt{C})-1)(1+o(1)),\label{T2_OPT}
   \end{eqnarray}
%    where the $o(1)$ terms above tend to zero as $n\rightarrow\infty$.
 and the corresponding optimal Bayes risk, denoted $R_{Opt}^{BO}$, is given by,
   \begin{equation}\label{OPT_BAYES_RISK}
    R_{Opt}^{BO}=n((1-p)t_1^{BO}+pt_2^{BO})=np(2\Phi(\sqrt{C})-1)(1+o(1)).
   \end{equation}

Under this set up, \citet{BCFG2011} introduced the notion of asymptotic Bayes optimality under sparsity (ABOS) defined as follows.
\begin{definition}
 A multiple testing procedure with Bayes risk $R$ with respect to the two-groups model (\ref{TWO_GROUP_X}), is said to be ABOS, if
 \begin{equation}
  \frac{R}{R_{Opt}^{BO}} \rightarrow 1 \mbox{ as } n\rightarrow\infty,\label{ABOS_DEFN}
 \end{equation}
where $R_{Opt}^{BO}$ denotes the optimal Bayes risk as given by (\ref{OPT_BAYES_RISK}) and the sequence of vectors $\{(\psi_n^2,p_n)\}_{n\geqslant1}$ is assumed to satisfy the conditions of Assumption \ref{ASSUMPTION_ASYMP}.
\end{definition}
In other words, a multiple testing procedure is said to be ABOS if within the asymptotic framework Assumption \ref{ASSUMPTION_ASYMP}, it attains the optimal Bayes risk in (\ref{OPT_BAYES_RISK}) up to the correct constant when the number of tests grows to infinity. \citet{BCFG2011} provided necessary and sufficient conditions for a fixed threshold multiple testing rules to be ABOS and established that the step-up multiple testing procedure due to \citet{BH1995} and the Bonferroni procedure are ABOS under very general conditions.\newline

For the independent normal means testing problem, \citet{CPS2010} observed through simulations that, under the assumption of sparsity, the posterior probability of the $i$-th alternative hypothesis being true under their two-groups formulation, can be well approximated by the posterior shrinkage weight $1-\widehat{\kappa}_i$, where $\widehat{\kappa}_i$ denotes the $i$-th posterior shrinkage coefficient based on the horseshoe prior. Looking at the close proximity of these two posterior quantities, \citet{CPS2010} proposed a natural thresholding rule under a symmetric $0-1$ loss based on the horseshoe prior, given by:
  \begin{equation}
  \mbox{ reject } H_{0i} \mbox{ if } 1-\widehat{\kappa}_i > 0.5, \mbox{ $i=1,\dots,n$}.\nonumber
 \end{equation}  

 \citet{CPS2010} numerically observed that under the assumption of sparsity, their proposed multiple testing procedure closely mimics the performance of the corresponding optimal Bayes rule based on their two-groups formulation which was later theoretically validated by \citet{DG2013}. For the multiple testing problem as in (\ref{TO_TEST}), \citet{DG2013} considered the following decision rule based on the horseshoe  prior assuming a symmetric $0-1$ loss, given by
  \begin{equation}\label{INDUCED_DECISION}
  \mbox{ reject } H_{0i} \mbox{ if } 1-E(\kappa_i|X_i,\tau) > 0.5, \mbox{ $i=1,\dots,n$},
 \end{equation}  
where the data is assumed to be generated according to the two-groups model (\ref{TWO_GROUP_X}). They showed that within the asymptotic framework of \citet{BCFG2011}, if $\tau\sim p$, the induced decisions in (\ref{INDUCED_DECISION}) attains the optimal Bayes risk in (\ref{OPT_BAYES_RISK}) up to a multiplicative constant, the constant being close to 1. In a more recent article, \citet{GTGC2015} generalized and improved the results of \citet{DG2013} over a broad class of one-group shrinkage priors that includes the horseshoe in particular. They considered a general family of one-group tail robust prior distributions where the prior distribution $\pi(\lambda_i^2)$ for the local shrinkage parameters $\lambda_i^2$ is given by (\ref{LAMBDA_PRIOR}) and satisfies the following:
 \begin{enumerate}[(I)]
   \item $\frac{1}{2} < a < 1$
   \item  $a= \frac{1}{2}$ and $L(t)/\sqrt{\log(t)} \rightarrow 0$ as $t \rightarrow \infty$.
   \end{enumerate} 

\citet{GTGC2015} showed that within their chosen asymptotic framework, if $\lim_{n\rightarrow\infty} \tau/p \in(0,\infty)$, then the Bayes risk of the multiple testing rules in (\ref{INDUCED_DECISION}) based on the above general class of prior distributions, denoted $R_{OG}$, satisfies
  \begin{equation}\label{BAYES_RISK_OG_GTGC}
   np\big[2\Phi\big(\sqrt{2a}\sqrt{C}\big)-1\big]\big(1+o(1)\big)
  \leqslant   
   R_{OG}
   \leqslant np\big[2\Phi\bigg(\sqrt{\frac{2aC}{\eta(1-\delta)}}\bigg)-1\big]\big(1+o(1)\big) \mbox{ as } n\rightarrow\infty,
  \end{equation}
 for every fixed $\eta \in (0,\frac{1}{2})$ and $\delta \in (0,1)$, where the $o(1)$ terms above are not necessarily the same, tend to zero as $n\rightarrow\infty$ and depend on the choice of $\eta \in (0,\frac{1}{2})$ and $\delta \in (0,1)$.\newline
 
In case the proportion $p$ is unknown, \citet{GTGC2015} considered a data adaptive procedure by replacing $\tau$ by an empirical Bayes estimate $\widehat{\tau}$ in the definition of the induced decisions in (\ref{INDUCED_DECISION}). The aforesaid empirical Bayes estimate of $\tau$ was proposed by \citet{PKV2014} and is given by,
\begin{equation}\label{TAU_HAT_EMPB}
 \widehat{\tau}= \max\bigg\{\frac{1}{n}, \frac{1}{c_2n}\sum_{j=1}^{n}1\{|X_j|> \sqrt{c_1\log n}\}\bigg\}
\end{equation}
where $c_1 \geq 2$ and $c_2 \geq 1$ are some predetermined finite real numbers. Letting $E(1-\kappa_i|X_i, \widehat{\tau})$ denote the $i-$th posterior shrinkage weight $E(1-\kappa_i|X_i,\tau)$ evaluated at $\tau=\widehat{\tau}$, \citet{GTGC2015} considered the following empirical Bayes procedure, given by,
\begin{equation}\label{INDUCED_DECISION_EB}
\mbox{ reject } H_{0i} \mbox{ if } 1-E(\kappa_i|X_i,\widehat{\tau}) > 0.5, \mbox{ $i=1,\dots,n$}.
\end{equation}

It was shown in \citet{GTGC2015} that, within the asymptotic framework of \citet{BCFG2011}, if $p\equiv p_n \propto n^{-\beta}$, for $\beta\in (0,1)$, the Bayes risk of the empirical Bayes multiple testing rules in (\ref{INDUCED_DECISION_EB}) above, denoted $R^{EB}_{OG}$, is bounded above by,
  \begin{equation}\label{EMP_BAYES_RISK_GTGC}
   R^{EB}_{OG}
   \leq np\big[2\Phi\bigg(\sqrt{\frac{2aC}{\eta(1-\delta)}}\bigg)-1\big]\big(1+o(1)\big) \mbox{ as } n\rightarrow\infty,
 \end{equation}
 for every fixed $\eta \in (0,\frac{1}{2})$ and $\delta \in (0,1)$, where the $o(1)$ term above tends to zero as $n\rightarrow\infty$ and depends on the choice of $\eta$ and $\delta$.\newline
  
Thus, the induced decisions proposed by \citet{CPS2010} based on a general family of heavy tailed one-group prior distributions, asymptotically attain the Oracle risk (\ref{OPT_BAYES_RISK}) up to $O(1)$, with the constant in the one-group risk being close to that in the Oracle risk. However, an interesting question which naturally arises is whether it is possible for such induced decisions to attain the optimal Bayes risk in (\ref{OPT_BAYES_RISK}) up to the correct constant, that is, whether such induced multiple testing rules can be asymptotically Bayes optimal when the hyperparameters $(\psi_n^2, p_n)$ of the two-groups model satisfy Assumption \ref{ASSUMPTION_ASYMP}. For that we consider in this section the multiple testing rules (\ref{INDUCED_DECISION}) and (\ref{INDUCED_DECISION_EB}) imposed by our chosen general class of one-group shrinkage priors, where the prior distribution for the local shrinkage parameters $\lambda_j^2$'s is given by (\ref{LAMBDA_PRIOR}) and the corresponding slowly varying component $L(\cdot)$ satisfies Assumption \ref{ASSUMPTION_LAMBDA_PRIOR}. It will be seen in the forthcoming subsections that the answer to the aforesaid question of asymptotic Bayes optimality is indeed in the affirmative for the horseshoe-type prior distributions.\newline 

But before going into the theoretical details of this section, let us first have a look at where we gain in our present approach as compared to those in \citet{DG2013} and \citet{GTGC2015}. Note that, a careful inspection of the proofs of the last two articles shows that, for the induced decisions under study, the contribution due to erroneously rejecting the true nulls to the overall Bayes risk is negligible compared to that due to erroneous acceptance of false nulls. Hence, it is the contribution due to the type II errors only where one can improve upon. Note that, the authors of \citet{DG2013} and \citet{GTGC2015} employed the following inequality:
 \begin{equation}
  E(\kappa_i|X_i,\tau)\leqslant \eta + \Pr(\kappa_i>\eta|X_i,\tau), \mbox{ for any } \eta\in(0,1).\nonumber
 \end{equation}
which played an important role for obtaining a non-trivial upper bound to the corresponding type II errors. On the other hand, in the present article, we use one key lemma involving the term $E(\kappa_i|X_i,\tau)$ (a version of which turns to be essential for proving the asymptotic minimaxity and related contraction results of Section 3), followed by some novel and delicate arguments, which to the best of our knowledge, has not been reported elsewhere before. The aforesaid lemma describes an important asymptotic behavior of the posterior quantity $E(\kappa_i|x,\tau)$ for large $x$'s, when both $x$ goes to infinity and and $\tau$ tends to 0 simultaneously at an appropriate rate, which cannot be explained using the above inequality. This is where we differ distinctively in our present approach compared to those of \citet{DG2013} and \citet{GTGC2015}. Moreover, using the aforesaid lemma, we also obtain a sharp asymptotic lower bound to the corresponding type I error probabilities when $\tau$ is treated as a tuning parameter, which, in turn, helps to deduce the fact that for the horseshoe-type priors, the optimal choice of $\tau$ should be asymptotically of the order of $p$ when it is treated as a tuning parameter only. This will be made more precise in Remark \ref{OPT_CHOICE_TAU} of this paper.

 \subsection{Asymptotic Bounds on Error Probabilities of Both Kinds}
Before studying the asymptotic risk properties of any multiple testing procedure, it is first necessary to investigate the asymptotic behaviors of the corresponding type I and type II error probabilities of the multiple testing procedure under study. In this section, we present four important results, namely, Theorem \ref{THM_T1_UB_LB} - Theorem \ref{THM_T2_EB_UB}, involving asymptotic bounds to the probabilities of type I and type II errors of the individual induced decisions, both when $\tau$ is treated as a tuning parameter and it is replaced by the empirical Bayes estimate $\widehat{\tau}$ defined in (\ref{TAU_HAT_EMPB}). While Theorem \ref{THM_T1_UB_LB} and Theorem \ref{THM_T2_UB_LB} below give asymptotic bounds to the error probabilities of both kinds of the $i$-th decision in (\ref{INDUCED_DECISION}), Theorem \ref{THM_T1_EB_UB} and Theorem \ref{THM_T2_EB_UB} give asymptotic upper bounds for the probabilities of type I and type II errors of the $i$-th empirical Bayes decision in (\ref{INDUCED_DECISION_EB}), respectively. These results are of fundamental importance for analyzing the behavior of the multiple testing procedures (\ref{INDUCED_DECISION}) and (\ref{INDUCED_DECISION_EB}) in terms of their corresponding Bayes risks under the two groups normal mixture model (\ref{TWO_GROUP_X}) and the usual additive loss. It would be worth noting in this context that the asymptotic upper bound to the type I error probability ($t_{1i}$) as given in Theorem \ref{THM_T1_UB_LB} and the asymptotic lower bound to the type II error probability ($t_{2i}$) as given by Theorem \ref{THM_T1_UB_LB}, are simple consequences of Theorem 4.4 and Theorem 4.7 of \citet{GTGC2015}, while the corresponding proofs for the asymptotic upper and lower bounds for $t_{2i}$ and $t_{1i}$, respectively, require some novel arguments as already discussed at the end of the preceding section. Using this novel argument, we obtain sharp asymptotic upper and lower bounds to the probabilities of type II and type I errors of the $i-th$ induced decision in (\ref{INDUCED_DECISION}) as given by Theorem \ref{THM_T2_UB_LB} and Theorem \ref{THM_T1_UB_LB}, respectively, which cannot be improved further. This will be made more precise later in this paper. Moreover, proofs of Theorem \ref{THM_T1_EB_UB} and Theorem \ref{THM_T2_EB_UB} follow using Theorem \ref{THM_T1_UB_LB} and Theorem \ref{THM_T2_UB_LB}, together with the arguments of \citet{GTGC2015}. Hence, we only state these results and excluded their proofs for the sake of brevity.

 \begin{thm}\label{THM_T1_UB_LB}
 Suppose $X_1,\dots,X_n$ are i.i.d. observations generated according to the two-groups normal mixture model (\ref{TWO_GROUP_X}) and suppose Assumption \ref{ASSUMPTION_ASYMP} is satisfied by the sequence of vectors $(\psi^2, p)$ defining the two-groups model (\ref{TWO_GROUP_X}). Suppose we wish to test simultaneously $H_{0i}:\nu_i=0$ vs $H_{Ai}:\nu_i=1$, for $i=1,\dots,n$, using the classification rule (\ref{INDUCED_DECISION}) induced by the general class of one-group shrinkage priors where the prior distribution of the local shrinkage parameter $\pi(\lambda_i^2)$ is given by (\ref{LAMBDA_PRIOR}) with $a\in (0,1)$, and the corresponding slowly varying component $L(\cdot)$ satisfies Assumption \ref{ASSUMPTION_LAMBDA_PRIOR}. Suppose $\tau =\tau_n \rightarrow 0$ as $n\rightarrow\infty$. Let us fix any $\eta\in(0,1)$ and any $\delta\in(0,1)$. Then the probability $t_{1i}$ of type I error of the $i$-th decision in (\ref{INDUCED_DECISION}) satisfies 
 \begin{equation}
  G_{1}(a,\eta,\delta)\frac{\big(\tau^{2a}\big)^{\frac{\zeta}{2}}}{\sqrt{\log(\frac{1}{\tau^2})}}(1+o(1))
 \leqslant t_1 \equiv t_{1i} \leqslant H_{1}(a,\eta,\delta)\frac{\tau^{2a}}{\sqrt{\log(\frac{1}{\tau^2})}}(1+o(1))\mbox{ as $n\rightarrow\infty$},\nonumber
 \end{equation}
  for any fixed $\zeta > \frac{2}{\eta(1-\delta)}$, where the $o(1)$ terms above do not depend on $i$ and are not equal. Moreover, the $o(1)$ terms appearing on the left hand side of the above inequality depends on the choice of $\eta \in (0,1)$ and $\delta \in (0,1)$, while the $o(1)$ term on the right hand side of the above inequality is independent of the choices of $\eta \in (0,1)$ and $\delta \in (0,1)$. Here $G_{1}(a,\eta,\delta)$ and $H_{1}(a,\eta,\delta)$ are some finite positive constants each being independent of both $i$ and $m$, but depend on $a\in(0,1)$, $\eta \in (0,1)$ and $\delta \in (0,1)$.
   \end{thm}
   
 \begin{proof}
  See Appendix.
 \end{proof}
  
\begin{thm}\label{THM_T2_UB_LB}
 Consider the set-up of Theorem \ref{THM_T1_UB_LB}. Let us assume that $\tau=\tau_n\rightarrow 0$ as $n\rightarrow\infty$ such that $\lim_{n\rightarrow\infty}\frac{\tau}{p^{\alpha}}\in(0,\infty)$, for $\alpha>0$. Then, for every fixed $\eta \in (0,1)$ and every fixed $\delta \in (0,1)$, the probability $t_{2i}$ of type II error of the $i$-th individual decision in (\ref{INDUCED_DECISION}) satisfies
    \begin{equation}
  \big[2\Phi(\sqrt{2a\alpha}\sqrt{C})-1\big](1+o(1))\leqslant t_2 \equiv t_{2i} \leqslant \big[2\Phi(\sqrt{\zeta a\alpha}\sqrt{C})-1\big](1+o(1)) \mbox { as } n\rightarrow\infty,\nonumber
   \end{equation}
 for any fixed $\zeta > \frac{2}{\eta(1-\delta)}$, where the $o(1)$ terms above do not depend on $i$ and are not identical. Moreover, the $o(1)$ term on the right hand side of the above inequality depends on the choice of $\eta \in (0,1)$ and $\delta \in (0,1)$, while the $o(1)$ term on the left hand side of the above inequality is independent of the choices of $\eta \in (0,1)$ and $\delta \in (0,1)$.
   \end{thm}
    \begin{proof}
  See Appendix.
 \end{proof}
 
 \begin{thm}\label{THM_T1_EB_UB}
 Suppose $X_1,\dots,X_n$ are i.i.d. observations generated according to the two-groups normal mixture model (\ref{TWO_GROUP_X}) and suppose Assumption \ref{ASSUMPTION_ASYMP} is satisfied by the sequence of vectors $(\psi^2, p)$ defining the two-groups model (\ref{TWO_GROUP_X}), with $p \propto n^{-\beta}$, for some $0<\beta<1$. Suppose we wish to test simultaneously $H_{0i}:\nu_i=0$ vs $H_{Ai}:\nu_i=1$, for $i=1,\dots,n$, using the classification rule (\ref{INDUCED_DECISION_EB}) induced by the general class of one-group shrinkage priors where the prior distribution of the local shrinkage parameter $\pi(\lambda_i^2)$ is given by (\ref{LAMBDA_PRIOR}) with $a\in (0,1)$, and the corresponding slowly varying component $L(\cdot)$ satisfies Assumption \ref{ASSUMPTION_LAMBDA_PRIOR}. Then, the probability $\widetilde{t}_{1i}$ of type I error of the $i$-th empirical Bayes decision in (\ref{INDUCED_DECISION_EB}) satisfies 
   \begin{equation}
   \widetilde{t}_{1i} \leq \frac{B_{1}^{*}\alpha_n^{2a}}{\sqrt{\log (\frac{1}{\alpha_n^2})}}(1+o(1))+ \frac{1/\sqrt{\pi}}{n^{c_1/2}\sqrt{\log n}} + e^{-2(2\log 2 -1)\beta_{0} np(1+o(1))}\mbox { as } n\rightarrow\infty,\nonumber
   \end{equation}
 where the $o(1)$ terms above are independent of $i$. Here $B^{*}_1$ and $\beta_{0}$ are some finite positive constants, each being independent of both $i$ and $n$, while $\alpha_n=\Pr(|X_1|>\sqrt{c_1\log n})$ depends on $n$ only.
   \end{thm}

 \begin{thm}\label{THM_T2_EB_UB}
 Let us consider the set-up of Theorem \ref{THM_T1_EB_UB}. Then, for every fixed $\eta \in (0,1)$ and every fixed $\delta \in (0,1)$, the probability $\widetilde{t}_{2i}$ of type II error of the $i$-th empirical Bayes decision in (\ref{INDUCED_DECISION_EB}) satisfies
   \begin{eqnarray}
   \widetilde{t}_{2i} \leq \big[2\Phi\big(\sqrt{\zeta a}\sqrt{C}\big)-1\big]\big(1+o(1)\big) \mbox { as } n\rightarrow\infty,\nonumber
   \end{eqnarray}
 for any fixed $\zeta > \frac{2}{\eta(1-\delta)}$. Here the $o(1)$ term is independent of $i$, but depends on the choices of $\eta \in (0,1)$ and $\delta \in (0,1)$.
   \end{thm}

 \subsection{Asymptotic Bayes Optimality Under Sparsity of Induced Testing Rules Based on Horseshoe-type Priors}
  In this section, we present in Theorem \ref{THM_BAYES_RISK_EXACT_ASYMP_ORDER} the exact asymptotic order of the ratio of Bayes risk of the induced multiple testing procedure (\ref{INDUCED_DECISION}) under study to that of the optimal Bayes risk (\ref{OPT_BAYES_RISK}) for a wide range of values of the global shrinkage parameter $\tau$ depending on the proportion of true alternatives $p$, which immediately shows that, within the asymptotic framework of \citet{BCFG2011}, if $\tau$ is asymptotically of the order of $p$, the multiple testing rules imposed by the horseshoe-type priors are ABOS. For the data adaptive empirical Bayes procedure, since it is already known that the corresponding Bayes risk is within a constant factor of the Oracle risk in (\ref{OPT_BAYES_RISK}) asymptotically (see \citet{GTGC2015}), we focus on the more interesting situation when $a=0.5$, that is, we consider the horseshoe-type priors only in this case. In Theorem \ref{THM_ABOS_EBAYES_HSTYPE}, we establish that, within the asymptotic framework of \citet{BCFG2011}, if $p\propto n^{-\beta}$, $0<\beta<1$, the empirical Bayes decisions (\ref{INDUCED_DECISION_EB}) based on the horseshoe-type priors will be ABOS. Proof of Theorem \ref{THM_BAYES_RISK_EXACT_ASYMP_ORDER} is based on the asymptotic bounds for the corresponding type I and type II error probabilities as given by Theorem \ref{THM_T1_UB_LB} and Theorem \ref{THM_T2_UB_LB}, followed by certain subtle analytic arguments, while proof of Theorem \ref{THM_ABOS_EBAYES_HSTYPE} is based on analogous arguments together with the techniques used for proving Theorem 3.2 of \citet{GTGC2015} and hence, it is omitted.
  
  \begin{thm}\label{THM_BAYES_RISK_EXACT_ASYMP_ORDER}
  Let $X_1,\dots,X_n$, be i.i.d. observations having the two-groups normal mixture distribution (\ref{TWO_GROUP_X}) where the sequence of vectors $(\psi^2, p)$ satisfy the conditions of Assumption \ref{ASSUMPTION_ASYMP}. Suppose we wish to test the $n$ hypotheses $H_{0i}:\nu_i=0$ vs $H_{Ai}:\nu_i=1$, for $i=1,\dots,n$, simultaneously, using the classification rule (\ref{INDUCED_DECISION}) induced by the general class of one-group shrinkage priors where the prior distribution of the local shrinkage parameter $\pi(\lambda_i^2)$ is given by (\ref{LAMBDA_PRIOR}) with $a\in [0.5,1)$, and the corresponding slowly varying component $L(\cdot)$ satisfies Assumption \ref{ASSUMPTION_LAMBDA_PRIOR}. It is further assumed that $\tau \rightarrow 0$ as $n\rightarrow\infty$ such that $\lim_{n\rightarrow\infty} \tau/p^{\alpha} \in(0,\infty)$, for $\alpha \geqslant 1$. Then, the Bayes risk of the multiple testing rules in (\ref{INDUCED_DECISION}), denoted $R_{OG}$, satisfies
  \begin{equation}\label{EXACT_ASYMP_ORDER_RATIO_RISKS}
   \lim_{n\rightarrow\infty}\frac{R_{OG}}{R_{Opt}^{BO}}=\frac{2\Phi\big(\sqrt{2a\alpha}\sqrt{C}\big)-1}{2\Phi\big(\sqrt{C}\big)-1}.
  \end{equation}
  In particular, for $a=0.5$ and $\alpha=1$ we have,
    \begin{equation}
   \lim_{n\rightarrow\infty}\frac{R_{OG}}{R_{Opt}^{BO}}=1.\nonumber
  \end{equation}
  \end{thm}

  \begin{proof}
   See Appendix.
  \end{proof}
 
 Some important consequences of Theorem \ref{THM_BAYES_RISK_EXACT_ASYMP_ORDER} are the following. {\it First of all}, it gives an exact asymptotic expression of the Bayes risk $R_{OG}$ of the induced decisions in (\ref{INDUCED_DECISION}), for any $a\in[0.5,1)$ and any $\alpha\geqslant 1$. {\it Secondly}, it shows that for various choices of the global shrinkage parameter $\tau$, the induced decisions based on our general class of one-group priors, asymptotically attain the optimal Bayes risk in (\ref{OPT_BAYES_RISK}) up to a multiplicative factor, the multiplicative factor being close to 1, provided $\alpha$ is not too large compared to 1. {\it Thirdly and most importantly}, it says that, if the global shrinkage parameter $\tau$ is asymptotically of the order of the proportion of true alternatives $p$, the induced decisions (\ref{INDUCED_DECISION}) based on the horseshoe-type priors (for which $a=0.5$), asymptotically attain the optimal Bayes risk up to the correct constant, and hence, are ABOS. This not only sharpens the results of \citet{DG2013} and \citet{GTGC2015}, but at the same time provides an exact optimality result for a very broad class of one-group shrinkage priors in the context of multiple testing. Moreover, the above limiting expression in (\ref{EXACT_ASYMP_ORDER_RATIO_RISKS}) clearly shows that the limiting value of the ratio of Bayes risks is an increasing function of both $\alpha \geqslant 1$ and $a\in [0.5,1)$. Therefore, taking $\alpha=1$ and $a=0.5$ in (\ref{EXACT_ASYMP_ORDER_RATIO_RISKS}) asymptotically yields the smallest possible value of the ratio of such Bayes risks, namely, 1. This explains why the hyperparameter {\it a} in the definition of $\pi(\lambda_i^2)$ in (\ref{LAMBDA_PRIOR}), should be set at $a=0.5$ as a default choice compared to other values of $a\in (0.5,1)$ for the present multiple testing problem.

 \begin{remark}\label{OPT_CHOICE_TAU}
 Some important observations are to be made in this regard about the effect of choosing different values of $\tau$ depending on $p$, when $p$ is assumed to be known. Although some of these observations were already made in \citet{GTGC2015}, the present remark helps us understand such facts very clearly. For that we confine our attention to the class of horseshoe-type priors for which one has $a=0.5$. Note that, the type I and type II error probabilities of the $i$-th decision in (\ref{INDUCED_DECISION}), that is, $t_{1i}$ and $t_{2i}$, do not depend on $i$ and their common values are given by $t_1$ and $t_2$, respectively. Thus, the Bayes risk of the decision rules in (\ref{INDUCED_DECISION}) is given by $R_{OG}=mp(\frac{1-p}{p}t_1+t_2)$ (using (\ref{BAYES_RISK_GEN})). Let us now assume that $\tau$ is asymptotically of the order of $p^{\alpha}$, for some $\alpha>0$. Let us first consider the case when $0<\alpha<1$. Then, given $\alpha\in(0,1)$, there always exist some $\eta \in (0,1)$ and some $\delta \in (0,1)$, such that $0<\alpha<\eta(1-\delta)$. Let us now fix any $\zeta > 2/(\eta(1-\delta))$, such that $\alpha < 2/\zeta < \eta(1-\delta)$, that is, $\zeta\in (\frac{2}{\eta(1-\delta)},\frac{2}{\alpha})$. Under this condition, using Theorem \ref{THM_T1_UB_LB}, we obtain the following:
 \begin{equation}
%   = \frac{1}{p^{1-\zeta\alpha/2}\sqrt{\log\frac{1}{p}}} 
  \frac{t_1}{p}
  \gtrsim \frac{p^{\zeta\alpha/2-1}}{\sqrt{\log\frac{1}{p}}} \rightarrow \infty \mbox{ if } p_n \rightarrow 0 \mbox{ as } n \rightarrow \infty,\nonumber
 \end{equation}
whence we have $R_{OG}/R_{Opt}^{BO}\rightarrow\infty$ as $n\rightarrow\infty$ within the asymptotic framework of \citet{BCFG2011}. Thus, the aforesaid asymptotic Bayes optimality result fails to hold in such situations. Consequently, for the present multiple testing problem, values of $\tau$ such that $\lim_{n\rightarrow\infty}\tau/p^{\alpha}\in(0,\infty)$, for $0<\alpha<1$, are not recommended for the horseshoe-type priors. On the other hand, the limiting value of the ratio $R_{OG}/R_{Opt}^{BO}$ as given by (\ref{EXACT_ASYMP_ORDER_RATIO_RISKS}), is non-decreasing in $\alpha \geqslant 1$, thereby attaining its minimum at $\alpha=1$. Using the preceding observations, it follows that, for the horseshoe-type priors, the optimal choice of $\tau$ is asymptotically of the order of $p$, when the latter is assumed to be known. In all other cases where $0.5<a<1$, though this choice of $\tau$ is not optimal in a strict mathematical sense, it still yields a Bayes risk that is within a constant factor of the Oracle risk (\ref{OPT_BAYES_RISK}) asymptotically, the factor being close to 1. Hence, for the present multiple testing problem, $\tau$ asymptotically of the order of $p$ should be regarded as the optimal choice of $\tau$ when $p$ is assumed to be known. It should however be noted in this context that the optimal choices of $\tau$ for the horseshoe-type priors such as those considered in Theorem \ref{THM_3.4} of Section 3, are not the same for simultaneous testing and estimation of independent normal means, and differ by a logarithmic factor of the proportion of non-zero means only.
\end{remark}

The next theorem shows that, within the asymptotic framework of \citet{BCFG2011}, when $p_n \propto n^{-\beta}$ for $0<\beta<1$, the empirical Bayes induced testing procedures (\ref{INDUCED_DECISION_EB}) based on the horseshoe-type priors asymptotically attain the optimal Bayes risk in (\ref{OPT_BAYES_RISK}) up to the correct constant, and hence, are ABOS. As already mentioned before, proof of this theorem follows using the arguments used in the proofs of Theorem 3.2 of \citet{GTGC2015} and Theorem \ref{THM_BAYES_RISK_EXACT_ASYMP_ORDER} of the present paper, and hence it is skipped.
 
 \begin{thm}\label{THM_ABOS_EBAYES_HSTYPE}
  Suppose $X_1,\dots,X_n$ are i.i.d. observations generated according to the two-groups normal mixture model (\ref{TWO_GROUP_X}) and suppose Assumption \ref{ASSUMPTION_ASYMP} is satisfied by the sequence of vectors $(\psi^2, p)$ defining the two-groups model (\ref{TWO_GROUP_X}), with $p \propto n^{-\beta}$, for some $0<\beta<1$. Suppose we wish to test simultaneously $H_{0i}:\nu_i=0$ vs $H_{Ai}:\nu_i=1$, for $i=1,\dots,n$, using the classification rule (\ref{INDUCED_DECISION_EB}) induced by the general class of one-group shrinkage priors where the prior distribution of the local shrinkage parameter $\pi(\lambda_i^2)$ is given by (\ref{LAMBDA_PRIOR}) with $a=0.5$, and the corresponding slowly varying component $L(\cdot)$ satisfies Assumption \ref{ASSUMPTION_LAMBDA_PRIOR}. Then the Bayes risk of the empirical Bayes testing procedure (\ref{INDUCED_DECISION_EB}), denoted $R^{EB}_{OG}$, satisfies
   \begin{equation}\label{EB_ABOS}
    \lim_{n\rightarrow\infty}\frac{R^{EB}_{OG}}{R_{Opt}^{BO}}=1,
  \end{equation}
  that is, the corresponding empirical Bayes decisions (\ref{INDUCED_DECISION_EB}) will be ABOS.
   \end{thm}  
 
\section{Discussion}
We studied in this paper various theoretical properties of a general class of heavy-tailed continuous shrinkage priors in terms of the quadratic minimax risk for estimating a multivariate normal mean vector which is known to be sparse in the sense of being nearly black. It is shown that Bayes estimators arising out of this general class asymptotically attain the minimax risk with respect to the $l_2$ norm, possibly up to some multiplicative constants. In particular, it is shown that for the horseshoe-type priors (defined in Section 2) such as the horseshoe, the Strawderman$-$Berger and the standard double Pareto priors, the corresponding Bayes estimators are asymptotically minimax in the sense that they attain the corresponding minimax risk under the $l_2-$norm up to the correct constant. Optimal rate of posterior contraction of these prior distributions around the truth in terms of the corresponding quadratic minimax error rate has also been established. We provided a novel unifying theoretical treatment that holds for a very broad class of one-group shrinkage priors. Another major contribution of this work is to show that shrinkage priors which are appropriately heavy tailed are good enough in order to attain the minimax optimal rate of contraction and that one doesn't need a pole at the origin, provided that the global tuning parameter is carefully chosen. This provides a partial answer to the question raised in \citet{PKV2014} already discussed in the introduction. As we already mentioned in Section 3 that one possible reason for such good performance of the kind of one-group shrinkage priors studied in this paper, is their ability to shrink the noise observations back to the origin, while leaving the large signals mostly unshrunk. Moreover, choice of the hyperparameter $a$ also plays a significant role for optimal posterior contraction of these priors. We believe that the theoretical results in this paper can be extended further for a more general class of one-group priors through exploiting properties of general slowly varying functions such as those considered in \citet{GTGC2015}. It would be worth mentioning that the concentration and moment inequalities of \citet{GTGC2015} and results like Lemma \ref{LEMMA_MOST_IMP} or Lemma \ref{MOST_IMP_LEMMA_ABOS} of the present paper, are extremely useful for analyzing and understanding various theoretical properties of such one-group shrinkage priors, and these results form the basis of both the asymptotic minimaxity and asymptotic Bayes optimality of the horseshoe-type priors under the assumption of sparsity. This makes us hopeful that the techniques employed in the present article would prove to be important ingredients for optimality studies of one-group continuous shrinkage priors.\newline

% We observed that (though not reported in this paper) when the number of non-zero means $q_n$ is unknown with $q_n\propto n^{\beta}$ for $0<\beta<1$, the Bayes estimators based on this general class of one-group priors, combined with the empirical Bayes estimate of the global shrinkage parameter as suggested in \citet{PKV2014}, still attain the minimax risk up to a multiplicative constant in the $l_2$ norm. This follows quite easily using Lemma \ref{LEMMA_MOST_IMP} and Lemma A.7 of \citet{PKV2014}. However

In the latter half of this paper, we also studied the asymptotic risk properties of induced decisions based on our general class of continuous shrinkage priors in the context of multiple testing within a Bayesian decision theoretic framework. A major theoretical contribution of this work is to show that within the asymptotic scheme of \citet{BCFG2011}, such induced decisions based on the horseshoe-type priors become asymptotically Bayes optimal under sparsity. To the best of our knowledge, this is the first such result in the Bayesian literature where the two-groups answer can be exactly achieved asymptotically by an one-group formulation under the assumption of sparsity. Another important contribution of the present work is to theoretically establish the fact that, when $a=0.5$, the optimal choice of the global variance component $\tau$ should be asymptotically of the order of the proportion of true alternatives $p$, when $p$ is assumed to be known. Moreover, the present work also provides strong theoretical support in favor of using $a=0.5$ as a default choice in our one-group prior specification. We hope that similar optimality results can be obtained for a more general class of one-group tail robust priors like those considered in \citet{GTGC2015} for the present multiple testing problem by carefully exploiting properties of slowly varying functions, together with the powerful concentration and moment inequalities of \citet{GTGC2015} and this way, one can relax the condition on the corresponding slowly varying component $L(\cdot)$ as given in Theorem 1 and Theorem 2 of \citet{GTGC2015} for the case $a=0.5$.\newline

Over the past few years, one-group shrinkage priors have been gaining increasing popularity in the Bayesian literature for modeling sparse high-dimensional data instead of the more natural two-groups model. However, not much was known about their various theoretical properties until very recently. To the best of our knowledge, \citet{BPPD2012} and \citet{DG2013} first studied certain asymptotic optimality properties of the Dirichlet$-$Laplace (DL) priors and the horseshoe prior, respectively, followed by the more recent works of \citet{PKV2014} and \citet{GTGC2015}. We hope that the present work is a useful contribution towards that end and provides important theoretical justifications from both frequentist and Bayesian view point in favor of the use of such kind of one-group priors together with some useful guidelines regarding the choice of the underlying hyperparameters while deciding over a one-group formulation in a given problem. However, an interesting problem that remains open till date is to show asymptotic optimality properties of a full Bayes approach by assigning a hyperprior to the global shrinkage parameter $\tau$. This applies in equal measure in both the problems of simultaneous testing and estimation. We hope to address this problem elsewhere in future.      

\appendix
\section*{Appendix}
\label{app}
\begin{lem}\label{LEM_MOMENT_INEQ}
For the general class of shrinkage priors (\ref{LAMBDA_PRIOR}) satisfying Assumption \ref{ASSUMPTION_LAMBDA_PRIOR} the following holds true for any $0 <  a < 1$:
  \begin{equation}\label{KAPPA_EXP}
 E(1-\kappa \big| x,\tau) \leqslant \frac{KM}{a(1-a)} e^{\frac{x^2}{2}}\tau^{2a}(1+o(1)), \mbox{ each fixed $x \in \mathbb{R}$,} \nonumber
  \end{equation}
 where $\kappa=\frac{1}{1+\lambda^2\tau^2}$ denote the shrinkage coefficients and the $o(1)$ term depends only on $\tau$ such that $\lim_{\tau \rightarrow 0}o(1)=0$.
\end{lem}
\begin{proof}
 See \citet{GTGC2015}.
\end{proof}

\begin{lem}\label{LEM_CONCENTRATION_INEQ}
For every fixed $\tau > 0$, and each fixed $\eta, \delta \in (0,1),$ the posterior distribution of the shrinkage coefficients $\kappa=1/(1+\lambda^2\tau^2)$ based on the general class of shrinkage priors (\ref{LAMBDA_PRIOR}) satisfying Assumption \ref{ASSUMPTION_LAMBDA_PRIOR}, with $a > 0$, satisfies the following concentration inequality:
  \begin{eqnarray}
   \Pr(\kappa > \eta|x,\tau) &\leqslant& \frac{H(a,\eta,\delta)e^{-\frac{\eta(1-\delta)x^{2}}{2}}}{\tau^{2a}\Delta(\tau^2,\eta,\delta)}, \mbox{ uniformly in } x \in \mathbb{R},\nonumber\\
   \nonumber\\
%   \end{eqnarray}
%   \begin{eqnarray}
  \mbox{where }  \Delta(\tau^2,\eta,\delta)&=&\xi(\tau^2,\eta,\delta)L\big(\frac{1}{\tau^2}(\frac{1}{\eta\delta}-1)\big),\nonumber\\ \nonumber\\
    \xi(\tau^2,\eta,\delta) &=& \frac{\int_{\frac{1}{\tau^2}\big(\frac{1}{\eta\delta}-1\big)}^{\infty}t^{-(a+\frac{1}{2}+1)}L(t)dt}{(a+\frac{1}{2})^{-1} \big(\frac{1}{\tau^2}\big(\frac{1}{\eta\delta}-1\big)\big)^{-(a+\frac{1}{2})}L(\frac{1}{\tau^2}\big(\frac{1}{\eta\delta}-1\big))},\quad  \mbox{and}\nonumber\\
    \nonumber\\
     H(a,\eta,\delta) &=& \frac{(a+\frac{1}{2}) (1-\eta\delta)^a}{ K(\eta\delta)^{(a+\frac{1}{2})}},\nonumber
  \end{eqnarray}
 where the term $\Delta(\tau^2,\eta,\delta)$ is such that $\lim_{\tau \rightarrow 0}\Delta(\tau^2,\eta,\delta)$ is a finite positive quantity for every fixed $\eta \in (0,1)$ and every fixed $\delta \in (0,1)$.
\end{lem}
\begin{proof}
 See \citet{GTGC2015}.
\end{proof}

We now prove an extremely important lemma which together with one of its variants, namely, Lemma \ref{MOST_IMP_LEMMA_ABOS}, form the basis of all the major theoretical results deduced in this paper.

% this paper.
% We shall use this lemma and  quite often for deriving the theoretical results given in Section 3 and Section 4 of this paper.
\begin{lem}\label{LEMMA_MOST_IMP}
Let us consider the general class of shrinkage priors (\ref{LAMBDA_PRIOR}) satisfying Assumption \ref{ASSUMPTION_LAMBDA_PRIOR}, with $a>0$. Then, for each fixed $\tau\in(0,1)$ and given any $c > 2$, the absolute difference between the Bayes estimators $T_{\tau}(x) $ based on the aforesaid class of shrinkage priors and an observation $x$, can be bounded above by a real valued function $h(\cdot,\tau)$, depending on $c$, and satisfying the following:
 
For any $\rho > c$,
\begin{eqnarray}
 \lim_{\tau \downarrow 0}\sup_{|x| > \sqrt{\rho \log\big(\frac{1}{\tau^{2a}}\big)}} h(x,\tau) = 0. \nonumber
 \end{eqnarray}
\end{lem}
\begin{proof}
By definition, 
\begin{eqnarray}
\mid T_{\tau}(x) - x\mid
&=& \mid xE(\kappa|x,\tau)\mid\nonumber\\
&=& \mid \frac{x\int_{0}^{1} \kappa\cdot \kappa^{a+\frac{1}{2}-1}(1-\kappa)^{-a-1}L(\frac{1}{\tau^2}(\frac{1}{\kappa}-1))e^{-\kappa x^2/2}d\kappa}{\int_{0}^{1} \kappa^{a+\frac{1}{2}-1}(1-\kappa)^{-a-1}L(\frac{1}{\tau^2}(\frac{1}{\kappa}-1))e^{-\kappa x^2/2}d\kappa}\mid\nonumber\\
&=& I(x,\tau), \mbox{ say.}\nonumber
\end{eqnarray}

Now observe that, for each fixed $\tau\in(0,1)$, the function $\mid T_{\tau}(x)-x\mid$ is symmetric in $x$ and it takes the value $0$ when $x=0$. Therefore, it would enough to find any non-negative function $h(x,\tau)$ that is symmetric in $x$ and satisfies the stated conditions. Hence, without any loss of generality, let us assume that $x>0$.\newline

Let us fix any $\eta \in (0,1) $ and any $\delta \in (0,1)$.\newline

Next, we observe that
\begin{equation}\label{IMP_DECOMPSOSITION}
 I(x,\tau) \leqslant I_{1}(x,\tau) + I_{2}(x,\tau)
\end{equation}
where $I_{1}(x,\tau)=\mid xE(\kappa 1\{\kappa < \eta \}\mid x,\tau)\mid$ and $I_{2}(x,\tau)=\mid xE(\kappa 1\{\kappa > \eta \} \mid x,\tau)\mid$.\newline

Now using the variable transformation $t=\frac{1}{\tau^2}(\frac{1}{\kappa}-1),$ we have the following:
\begin{eqnarray}
  I_{1}(x,\tau)
%  &=& \mid xE(\kappa 1\{\kappa < \eta \} \mid x,\tau)\mid\nonumber\\
 &=& \mid \frac{x\int_{0}^{\eta} \kappa\cdot \kappa^{a+\frac{1}{2}-1}(1-\kappa)^{-a-1}L(\frac{1}{\tau^2}(\frac{1}{\kappa}-1))e^{-\kappa x^2/2}d\kappa}{\int_{0}^{1} \kappa^{a+\frac{1}{2}-1}(1-\kappa)^{-a-1}L(\frac{1}{\tau^2}(\frac{1}{\kappa}-1))e^{-\kappa x^2/2}d\kappa}\mid\nonumber\\
 &=& \mid \frac{x\int_{\frac{1}{\tau^2}(\frac{1}{\eta}-1)}^{\infty} \frac{1}{(1+t\tau^2)^{3/2}}t^{-a-1}L(t)e^{-\frac{x^2}{2(1+t\tau^2)}}dt}{\int_{0}^{\infty}\frac{1}{(1+t\tau^2)^{1/2}}t^{-a-1}L(t)e^{-\frac{x^2}{2(1+t\tau^2)}}dt}\mid\nonumber \\
 &\leqslant& \mid \frac{x\int_{\frac{1}{\tau^2}(\frac{1}{\eta}-1)}^{\infty}\frac{1}{(1+t\tau^2)^{3/2}}t^{-a-1}L(t)e^{-\frac{x^2}{2(1+t\tau^2)}}dt}{\int_{\frac{t_{0}}{\tau^2}}^{\infty}\frac{1}{(1+t\tau^2)^{1/2}}t^{-a-1}L(t)e^{-\frac{x^2}{2(1+t\tau^2)}}dt}\mid \mbox{}\nonumber\\
 &=& J_{1}(x,\tau)\mbox{ say,} \label{IMP_DECOMPSOSITION_1_1}
\end{eqnarray}

Next observe that $\frac{t_{0}}{\tau^2} > t_{0}$ as $\tau^2 < 1$. Hence $L(t) \geqslant c_{0}$ for every $t \geqslant \frac{t_{0}}{\tau^2}$. Also, the function $L$ is bounded by the constant $M >0$. Utilizing these two observations and using the variable transformation $u=\frac{x^2}{1+t\tau^2}$ in both the numerator and the denominator of $J_{1}(x,\tau)$ in (\ref{IMP_DECOMPSOSITION_1_1}), and writing $s=\frac{1}{1+t_{0}} \in (0,1)$, we see that the term $J_{1}(x,\tau)$ can be bounded above as follows:
\begin{eqnarray}
J_{1}(x,\tau)
&\leqslant& \frac{M}{c_{0}} \mid x \frac{\int_{0}^{\eta x^2}e^{-u/2}\big(\frac{u}{x^2}\big)^{3/2}\big(\frac{1}{\tau^2}\big(\frac{x^2}{u}-1\big)\big)^{-a-1} \frac{x^2}{\tau^2u^2}du}{\int_{0}^{sx^2}e^{-u/2}\big(\frac{u}{x^2}\big)^{1/2}\big(\frac{1}{\tau^2}\big(\frac{x^2}{u}-1\big)\big)^{-a-1}\frac{x^2}{\tau^2u^2}du} \mid \nonumber\\
&=& \frac{M}{c_{0}} \mid \frac{1}{x} \cdot \frac{\int_{0}^{\eta x^2}e^{-u/2}u^{a+3/2-1}\big(1-\frac{u}{x^2}\big)^{-a-1}du}{\int_{0}^{s x^2}e^{-u/2}u^{a+1/2-1}\big(1-\frac{u}{x^2}\big)^{-a-1}du} \mid\nonumber
\end{eqnarray}
Note that when $0 < u < \eta x^2$ we have $0 < \frac{u}{x^2} < \eta < 1$, that is, $ 1-\eta < 1 -\frac{u}{x^2} < 1$. Similarly, we have $1-s < 1 -\frac{u}{x^2} < 1$ when $0 < u < s x^2$. Using these observations we obtain,
\begin{eqnarray}
J_{1}(x,\tau)
 &\leqslant& \frac{M}{c_{0}(1-\eta)^{1+a}} \mid \frac{1}{x} \cdot \frac{\int_{0}^{\eta x^2}e^{-u/2}u^{a+3/2-1}du}{\int_{0}^{s x^2}e^{-u/2}u^{a+1/2-1}du} \mid  \nonumber\\
 &\leqslant& \frac{M}{c_{0}(1-\eta)^{1+a}} \mid \frac{1}{x} \cdot \frac{\int_{0}^{\infty}e^{-u/2}u^{a+3/2-1}du}{\int_{0}^{s x^2}e^{-u/2}u^{a+1/2-1}du} \mid  \nonumber\\
 &=& h_{1}(x,\tau)\mbox{ say},\label{IMP_DECOMPSOSITION_1_2}
\end{eqnarray}
where $h_{1}(x,\tau)=C_{\textasteriskcentered} \big[\mid x \int_{0}^{s x^2}e^{-u/2}u^{a+1/2-1}du \mid\big]^{-1}  $ for some global constant $C_{\textasteriskcentered} \equiv C_{\textasteriskcentered}(a,\eta,L) > 0$ which is independent of both $x$ and $\tau$. Note that the function $h_{1}(x,\tau)$ is actually independent of $\tau$ and depends on $x$ only.\newline

Next we observe that,
\begin{eqnarray}
I_{2}(x,\tau)
&=&\mid xE(\kappa 1\{\kappa > \eta \} \mid x,\tau) \mid \nonumber\\
% &\leqslant& \mid \frac{x\int_{\eta}^{1} \kappa^{a+\frac{1}{2}-1}(1-\kappa)^{-a-1}L(\frac{1}{\tau^2}(\frac{1}{\kappa}-1))e^{-\kappa x^2/2}d\kappa}{\int_{0}^{1} \kappa^{a+\frac{1}{2}-1}(1-\kappa)^{-a-1}L(\frac{1}{\tau^2}(\frac{1}{\kappa}-1))e^{-\kappa x^2/2}d\kappa}\mid\nonumber\\
&\leqslant& \mid x\Pr(\kappa > \eta \mid x,\tau) \mid\nonumber\\
&\leqslant& \mid x\frac{H(a,\eta,\delta)e^{-\frac{\eta(1-\delta) x^2}{2}}}{\tau^{2a}\Delta(\tau^2,\eta,\delta)} \mid \nonumber\\
&=& h_{2}(x,\tau)\mbox{ say},\label{IMP_DECOMPSOSITION_2}
\end{eqnarray}
 
Let $h(x,\tau) = h_{1}(x,\tau)+h_{2}(x,\tau)$. Therefore combining (\ref{IMP_DECOMPSOSITION}), (\ref{IMP_DECOMPSOSITION_1_1}), (\ref{IMP_DECOMPSOSITION_1_2}) and (\ref{IMP_DECOMPSOSITION_2}), we finally obtain for every $x \in \mathbb{R} $ and $\tau\in(0,1)$,
\begin{equation}
 \mid T_{\tau}(x)-x \mid \leqslant h(x,\tau).\label{UB_FUNCTION}
\end{equation}

Note that the function $h(x,\tau)$ defined above is symmetric in $x$ about the origin. Now observe that the function $h_{1}(x,\tau)$ is strictly decreasing in $|x|$. Therefore, for any fixed $\tau\in(0,1)$ and every $\rho > 0$,
\begin{eqnarray}
 \sup_{|x| > \sqrt{\rho \log\big(\frac{1}{\tau^{2a}}\big)}} h_{1}(x,\tau)
 &\leqslant& C_{\textasteriskcentered} \big[\mid \sqrt{\rho \log\big(\frac{1}{\tau^{2a}}\big)} \int_{0}^{s\rho \log\big(\frac{1}{\tau^{2a}}\big)}e^{-u/2}u^{a+1/2-1}du \mid\big]^{-1} \nonumber
\end{eqnarray}
implying that 
\begin{equation}\label{LIMSUP_1}
 \lim_{\tau \downarrow 0}\sup_{|x| > \sqrt{\rho \log\big(\frac{1}{\tau^{2a}}\big)}} h_{1}(x,\tau) = 0.
\end{equation}

Again the function $h_{2}(x,\tau)$ is eventually decreasing in $|x|$. Therefore, for all sufficiently small $\tau\in(0,1)$,
\begin{equation}
 \sup_{|x| > \sqrt{\rho \log\big(\frac{1}{\tau^{2a}}\big)}} h_{2}(x,\tau) \leqslant h_{2}(\sqrt{\rho \log\big(\frac{1}{\tau^{2a}}\big)},\tau).\nonumber
\end{equation}

Let $\beta\equiv \beta(\eta,\delta) = \lim_{\tau \rightarrow 0} \Delta(\tau^2,\eta,\delta)$ for every fixed $\eta\in(0,1)$ and every fixed $\delta \in (0,1)$. Then $ 0 < \beta < \infty$ which follows from Lemma \ref{LEM_CONCENTRATION_INEQ}. Then, 
\begin{eqnarray}
\lim_{\tau \rightarrow 0} h_{2}(\sqrt{\rho \log\big(\frac{1}{\tau^{2a}}\big)},\tau)
% &=& \lim_{\tau \rightarrow 0} \mid \frac{\tau^{-2a}x\exp\big(-\frac{\eta(1-\delta)x^2}{2}\big)}{\Delta(\tau^2,\eta,\delta)} \mid \nonumber\\
&=& \frac{1}{\beta} \lim_{\tau \rightarrow 0} \mid \tau^{-2a}\sqrt{\rho\log\big(\frac{1}{\tau^{2a}}\big)}e^{- \frac{\eta(1-\delta)}{2}\rho\log\big(\frac{1}{\tau^{2a}}\big)}\mid \nonumber\\
&=& \frac{\sqrt{\rho}}{\alpha} \lim_{\tau \rightarrow 0} \big(\tau^{2a})^{\frac{\eta(1-\delta)}{2}\big(\rho-\frac{2}{\eta(1-\delta)}\big)} \sqrt{\log\big(\frac{1}{\tau^{2a}}\big)}\nonumber\\
&=& \left\{ \begin{array}{rl}
0 &\mbox{ if $\rho > \frac{2}{\eta(1-\delta)} $} \\
\infty &\mbox{ otherwise,}
\end{array}\right.\nonumber
\end{eqnarray}

whence it follows that
\begin{eqnarray}
\lim_{\tau \rightarrow 0} \sup_{|x| > \sqrt{\rho \log\big(\frac{1}{\tau^{2a}}\big)}} h_{2}(x,\tau)
&=& \left\{ \begin{array}{rl}
0 &\mbox{ if $\rho > \frac{2}{\eta(1-\delta)} $} \\
\infty &\mbox{ otherwise.}
\end{array}\right.\label{LIMSUP_2}
\end{eqnarray}

Combining (\ref{LIMSUP_1}) and (\ref{LIMSUP_2}) together with the facts that 
$$\lim_{\tau \rightarrow 0} \sup_{|x| > \sqrt{\rho \log\big(\frac{1}{\tau^{2a}}\big)}} h(x,\tau) \leqslant \lim_{\tau \rightarrow 0} \sup_{|x| > \sqrt{\rho \log\big(\frac{1}{\tau^{2a}}\big)}} h_{1}(x,\tau) + \lim_{\tau \rightarrow 0} \sup_{|x| > \sqrt{\rho \log\big(\frac{1}{\tau^{2a}}\big)}} h_{2}(x,\tau) $$
and
$$\lim_{\tau \rightarrow 0} \sup_{|x| > \sqrt{\rho \log\big(\frac{1}{\tau^{2a}}\big)}} h(x,\tau) \geqslant \lim_{\tau \rightarrow 0} \sup_{|x| > \sqrt{\rho \log\big(\frac{1}{\tau^{2a}}\big)}} h_{2}(x,\tau) $$
it follows that
\begin{eqnarray}
\lim_{\tau \rightarrow 0} \sup_{|x| > \sqrt{\rho \log\big(\frac{1}{\tau^{2a}}\big)}} h(x,\tau)
&=& \left\{ \begin{array}{rl}
0 &\mbox{ if $\rho > \frac{2}{\eta(1-\delta)} $} \\
\infty &\mbox{ otherwise.}
\end{array}\right.\label{LIMSUP_FINAL}
\end{eqnarray}

Observe that by choosing $\eta$ appropriately close to 1 and $\delta$ close to 0, any real number larger than 2 can be expressed in the form $\frac{2}{\eta(1-\delta)}.$ For example, taking $\eta=\frac{5}{6}$ and $\delta=\frac{1}{5}$ we obtain $\frac{2}{\eta(1-\delta)}=3$. Hence, given any $c > 2$, let us choose $0 < \eta, \delta < 1 $ such that $c=\frac{2}{\eta(1-\delta)}$. Clearly, the choice of $h(\cdot,\tau)$ depends on $c$. This, coupled with (\ref{UB_FUNCTION}) and (\ref{LIMSUP_FINAL}), completes the proof of Lemma \ref{LEMMA_MOST_IMP}.
\end{proof}

\paragraph{Remark A.1}
The function $h(\cdot,\tau)$ defined in the proof of Lemma \ref{LEMMA_MOST_IMP} satisfies the following:

\begin{equation}
 \lim_{|x| \rightarrow\infty}h(x,\tau)=0, \mbox{ for each fixed $\tau\in(0,1)$}.\nonumber
\end{equation}

This means (using Lemma \ref{LEMMA_MOST_IMP}) that for any fixed $\tau\in(0,1)$, we have,
\begin{equation}
 \lim_{|x| \rightarrow\infty}\mid T_{\tau}(x)-x \mid=0. \label{SIGNALS_UNSHRUNK}
\end{equation}

Equation (\ref{SIGNALS_UNSHRUNK}) above shows that for the general class of tail robust priors under consideration, large observations almost remain unshrunk no matter however small $\tau$ is.

% \paragraph{Remark A.2}
% It should be noted that the choice of the function $h$ as in Lemma \ref{LEMMA_MOST_IMP} is not unique since it depends on $c >1$. Moreover, one can easily obtain some other function, say, $\widetilde{h}$, through exploiting the integrals used in the proof of Lemma \ref{LEMMA_MOST_IMP} in a different manner.

\paragraph{Proof of Theorem \ref{THM_3.1}}
\begin{proof}
Suppose that $X \sim \mathcal{N}_{n}(\theta,I_{n})$, $\theta \in l_{0}[q_{n}]$ and $\tilde{q}_n=\#\{i:\theta_i\neq 0\}$. Let us split $E_{\theta}||T_{\tau}(X)-\theta||^2=\sum\limits_{i=1}^{n} E_{\theta_{i}} \big(T_{\tau}(X_{i})-\theta_{i}\big)^2 $ into two parts as follows:
\begin{equation}\label{THM_3.1_MAIN_DECOMPOSITION}
\sum_{i=1}^{n} E_{\theta_{i}} \big(T_{\tau}(X_{i})-\theta_{i}\big)^2
= \sum_{i:\theta_i\neq 0} E_{\theta_{i}} \big(T_{\tau}(X_{i})-\theta_{i}\big)^2 + \sum_{i:\theta_i=0} E_{\theta_{i}} \big(T_{\tau}(X_{i})-\theta_{i}\big)^2
\end{equation}
We shall show that, for all sufficiently small $\tau\in(0,1)$, the mean square errors due to the non-zero means and the zero means as in the right hand side of (\ref{THM_3.1_MAIN_DECOMPOSITION}), can be bounded above by $\tilde{q}_n\log\big(\frac{1}{\tau^{2a}}\big)$ and $(n-\tilde{q}_n)\tau^{2a}\sqrt{\log\big(\frac{1}{\tau^{2a}}\big)}$, respectively, up to some multiplicative constants.\newline

{\it Non-zero means:}\newline

For $\theta_i\neq 0$, let us split the corresponding mean square error as follows:
\begin{eqnarray}
 E_{\theta_{i}}\big(T_{\tau}(X_i) - \theta_{i} \big)^2
 &=& E_{\theta_{i}}\big( \big(T_{\tau}(X_i) - X_i \big) + \big(X_i - \theta_{i}\big) \big)^2\nonumber\\
 &=& E_{\theta_{i}}\big(T_{\tau}(X_i) - X_i\big)^2+ E_{\theta_{i}}\big(X_i - \theta_{i}\big)^2+2E_{\theta_{i}}\big(T_{\tau}(X_i) - X_i\big)\big(X_i - \theta_{i}\big) \nonumber\\
 &\leqslant& E_{\theta_{i}}\big(T_{\tau}(X_i) - X_i\big)^2+1+2 \sqrt{E_{\theta_{i}}\big(T_{\tau}(X_i) - X_i\big)^2} \nonumber\\
 &=& \bigg[\sqrt{E_{\theta_{i}}\big(T_{\tau}(X_i) - X_i\big)^2}+1\bigg]^2 \label{THM_3.1_UB1}
\end{eqnarray}
where we use the Cauchy-Schwartz inequality and the fact $E_{\theta_{i}}\big(X_i - \theta_{i}\big)^2=1$ in the step preceding the final step in the above chain of inequalities.\newline

Let us now define
$$\zeta_{\tau}:=\sqrt{2\log\big(\frac{1}{\tau^{2a}}\big)}.$$

% Observe that the random variable $\big(T_{\tau}(X_i) - X_i\big)^2$ is always bounded above by $X_i^2$ and hence, $E_{\theta_{i}}\big(T_{\tau}(X_i) - X_i\big)^2 <\infty$, and this will be true for any $1 \leqslant i \leqslant m$.\\
% 
Let us now fix any $c>1$ and choose any $\rho>c$. Then, using Lemma \ref{LEMMA_MOST_IMP}, there exists a non-negative real-valued function $h(\cdot,\tau)$, depending on $c$, such that
 \begin{equation}
  |T_{\tau}(x)-x| \leqslant h(x,\tau), \mbox{ for all } x \in \mathbb{R}\label{THM_3.1_LEM_4.3_UB1}
\end{equation}
and
 \begin{equation}
  \lim_{\tau \downarrow 0}\sup_{|x| > \rho \zeta_{\tau}} h(x,\tau) = 0.\label{THM_3.1_LEM_4.3_UB2}
\end{equation}

Once again, using the fact $\big(T_{\tau}(X_i) - X_i\big)^2 \leqslant X_i^2$, together with (\ref{THM_3.1_LEM_4.3_UB1}), we obtain,
\begin{eqnarray}
 E_{\theta_{i}}\big(T_{\tau}(X_i) - X_i\big)^2
 &=& E_{\theta_{i}}\big[(T_{\tau}(X_i) - X_i\big)^2 1\big\{|X_i|\leqslant \rho\zeta_{\tau}\big\}\big]\nonumber\\
 && + \mbox{ } E_{\theta_{i}}\big[(T_{\tau}(X_i) - X_i\big)^2 1\big\{|X_i|> \rho\zeta_{\tau}\big\}\big]\nonumber\\
%  &\leqslant& E_{\theta_{i}}\big[(T_{\tau}(X_i) - X_i\big)^2\big||X_i|\leqslant \rho\zeta_{\tau}\big]+E_{\theta_{i}}\big[(T_{\tau}(X_i) - X_i\big)^2\big||X_i|> \rho\zeta_{\tau}\big]\nonumber\\
%  &\leqslant& E_{\theta_{i}}\big[X_i^2\big||X_i|\leqslant \rho\zeta_{\tau}\big]+\sup_{|x|> \rho\zeta_{\tau}}h^2(x,\tau)\nonumber\\
 &\leqslant& \rho^2\zeta_{\tau}^2+\bigg(\sup_{|x|> \rho\zeta_{\tau}}h(x,\tau)\bigg)^2 \label{THM_3.1_UB2}
\end{eqnarray}
Now using (\ref{THM_3.1_LEM_4.3_UB2}) and the fact that $\zeta_{\tau}\rightarrow\infty$ as $\tau\rightarrow 0$, it follows that,
\begin{equation}
 \bigg(\sup_{|x|> \rho\zeta_{\tau}}h(x,\tau)\bigg)^2=o(\zeta_{\tau}^2) \mbox{ as } \tau\rightarrow 0.\label{THM_3.1_h_eq_o_zeta}
\end{equation}
On combining (\ref{THM_3.1_UB2}) and (\ref{THM_3.1_h_eq_o_zeta}), it follows that,
\begin{equation}
 E_{\theta_{i}}\big(T_{\tau}(X_i) - X_i\big)^2 \leqslant \rho^2\zeta_{\tau}^2(1+o(1)), \quad\mbox{as }\tau \rightarrow 0. \label{THM_3.1_UB_EXPCT_DIFF_SQ}
\end{equation}
Note that (\ref{THM_3.1_UB_EXPCT_DIFF_SQ}) holds uniformly for any $i$ such that $\theta_i \neq 0$, whence we have,
\begin{equation}
 \sum_{i:\theta_i\neq 0} E_{\theta_{i}} \big(T_{\tau}(X_{i})-\theta_{i}\big)^2 \lesssim \tilde{q}_n\zeta_{\tau}^2, \quad \mbox{as }\tau \rightarrow 0.\label{THM_3.1_UB_NONZERO_THETA} 
\end{equation}
 
{\it Zero means:}\newline

For $\theta_i=0$, the corresponding mean square error is split as follows:
 \begin{equation}
 E_{0}\big[T_{\tau}(X_i)^{2}\big]=E_{0}\big[T_{\tau}(X_i)^{2}1\{|X_i|\leqslant \zeta_{\tau}\}\big]+E_{0}\big[T_{\tau}(X_i)^{2}1\{|X_i| >\zeta_{\tau}\}\big],\label{THM_3.1_EQ_PETZ_1}
 \end{equation}
 where $\zeta_{\tau}=\sqrt{2\log\big(\frac{1}{\tau^{2a}}\big)}.$\newline
 
 For the first term on the right hand side of (\ref{THM_3.1_EQ_PETZ_1}), denoting $g_1(a)=\frac{KM}{a(1-a)}$ and using Lemma \ref{LEM_MOMENT_INEQ}, we obtain,
 \begin{eqnarray}
  E_{0} T_{\tau}(X_i)^{2} 1\{|X_i|\leqslant \zeta_{\tau}\}
  &\leqslant& \frac{g_1(a)^2}{\sqrt{2\pi}} (\tau^{2a})^2 \int_{-\zeta_{\tau}}^{\zeta_{\tau}} x^2 e^{\frac{x^2}{2}}dx (1+o(1)) \mbox{ as } \tau \rightarrow 0 \nonumber\\
  &=& \frac{2g_1(a)^2}{\sqrt{2\pi}} (\tau^{2a})^2 \int_{0}^{\zeta_{\tau}} x^2 e^{\frac{x^2}{2}}dx (1+o(1)) \mbox{ as } \tau \rightarrow 0 \nonumber\\
  &\leqslant& \sqrt{\frac{2}{\pi}}g_1(a)^2(\tau^{2a})^2 \zeta_{\tau}\frac{1}{\tau^{2a}}(1+o(1)) \mbox{ as } \tau \rightarrow 0 \nonumber\\
  &\lesssim& \zeta_{\tau}\tau^{2a} \mbox{ as } \tau \rightarrow 0,\label{THM_3.1_EQ_PETZ_2}
 \end{eqnarray} 
 while for the second term using the fact that $|T_{\tau}(x)|\leqslant |x|$ for $x\in\mathbb{R}$ we have,
 \begin{eqnarray}
 E_{0} T_{\tau}(X_i)^{2} 1\{|X_i| > \zeta_{\tau}\}
%  &\leqslant& E_{0} X^{2} 1\{|X| > \zeta_{\tau}\}\nonumber\\
 &\leqslant& 2 \int_{\zeta_{\tau}}^{\infty} x^2\phi(x)dx\nonumber\\
 &=& 2 \big[\zeta_{\tau}\phi(\zeta_{\tau})+(1-\Phi(\zeta_{\tau})) \big]\nonumber\\
 &\leqslant& 2\zeta_{\tau}\phi(\zeta_{\tau})+2\frac{\phi(\zeta_{\tau})}{\zeta_{\tau}}\nonumber\\
%  &=& 2\zeta_{\tau}\phi(\zeta_{\tau})(1+o(1))\mbox{ as } \tau \rightarrow 0 \nonumber\\
 &=& \sqrt{\frac{2}{\pi}}\zeta_{\tau}\tau^{2a}(1+o(1))\mbox{ as } \tau \rightarrow 0 \nonumber\\
 &\lesssim& \zeta_{\tau}\tau^{2a} \mbox{ as } \tau \rightarrow 0. \label{THM_3.1_EQ_PETZ_3}
 \end{eqnarray}
Combining equations (\ref{THM_3.1_EQ_PETZ_1}), (\ref{THM_3.1_EQ_PETZ_2}) and (\ref{THM_3.1_EQ_PETZ_3}), it follows that,
\begin{eqnarray}
\sum_{i:\theta_i=0}E_{\theta_{i}}\big(T_{\tau}(X_i)-\theta_{i} \big)^2
&\lesssim& (n-\tilde{q}_n)\tau^{2a}\sqrt{\log\big(\frac{1}{\tau^{2a}}\big)} \mbox{ as }\tau \rightarrow 0. \label{THM_3.1_UB_ZERO_THETA}
\end{eqnarray}
Finally, on combining (\ref{THM_3.1_MAIN_DECOMPOSITION}), (\ref{THM_3.1_UB_NONZERO_THETA}) and (\ref{THM_3.1_UB_ZERO_THETA}), we have,
\begin{equation}
 \sum_{i=1}^{n} E_{\theta_{i}} \big(T_{\tau}(X_{i})-\theta_{i}\big)^2
 \lesssim \tilde{q}_n\log\big(\frac{1}{\tau^{2a}}\big)+ (n-\tilde{q}_n)\tau^{2a}\sqrt{\log\big(\frac{1}{\tau^{2a}}\big)}\mbox{ as } \tau \rightarrow 0.
\end{equation}
The stated result now becomes immediate by observing that $\tilde{q}_n\leqslant q_n$ and $q_n=o(n)$ and then taking supremum over all $\theta\in l_{0}[q_{n}]$. This completes the proof of Theorem \ref{THM_3.1}.
\end{proof}

\begin{lem}\label{POST_VAR_IDENTITY}
The posterior variance arising out of the general class of shrinkage priors (\ref{LAMBDA_PRIOR}) can be represented by the following identity:
 \begin{eqnarray}
  Var(\theta|x)
  &=&\frac{T_{\tau}(x)}{x} - \big(T_{\tau}(x) -x \big)^2 + x^2 \frac{\int_{0}^{\infty}\frac{1}{(1+t\tau^2)^{5/2}}t^{-a-1}L(t)e^{-\frac{x^2}{2(1+t\tau^2)}}dt}{\int_{0}^{\infty}\frac{1}{(1+t\tau^2)^{1/2}}t^{-a-1}L(t)e^{-\frac{x^2}{2(1+t\tau^2)}}dt}\label{VAR_IDNT_1}\\
   &=& \frac{T_{\tau}(x)}{x} - T_{\tau}^2(x) + x^2 \frac{\int_{0}^{\infty}\frac{(t\tau^2)^2}{(1+t\tau^2)^{5/2}}t^{-a-1}L(t)e^{-\frac{x^2}{2(1+t\tau^2)}}dt}{\int_{0}^{\infty}\frac{1}{(1+t\tau^2)^{1/2}}t^{-a-1}L(t)e^{-\frac{x^2}{2(1+t\tau^2)}}dt}\label{VAR_IDNT_2} 
 \end{eqnarray}
 which can be bounded from above by
 \begin{equation}
  Var(\theta|x) \leqslant 1 + x^2.\nonumber
 \end{equation}
\end{lem}

\begin{proof} By the law of iterated variance it follows that
\begin{eqnarray}
 Var(\theta|x)
 &=& E\big[Var(\theta|x,\kappa,\tau)\big]+ Var\big[E(\theta|x,\kappa,\tau)\big]\nonumber\\
 &=& E\big[(1-\kappa)|x,\tau\big]+Var\big[x(1-\kappa)|x,\tau)\big]\nonumber\\
 &=& E\big[(1-\kappa)|x,\tau\big]+ x^2 Var\big[\kappa|x,\tau)\big]\nonumber\\
 &=& E\big[(1-\kappa)|x,\tau\big]+ x^2 E\big[\kappa^2|x,\tau\big] - x^2E^2\big[\kappa|x,\tau\big]\nonumber\\
 &=& \frac{T_{\tau}(x)}{x} - \big(T_{\tau}(x) -x \big)^2 + x^2 \frac{\int_{0}^{\infty}\frac{1}{(1+t\tau^2)^{5/2}}t^{-a-1}L(t)e^{-\frac{x^2}{2(1+t\tau^2)}}dt}{\int_{0}^{\infty}\frac{1}{(1+t\tau^2)^{1/2}}t^{-a-1}L(t)e^{-\frac{x^2}{2(1+t\tau^2)}}dt}\nonumber
\end{eqnarray}
 which can equivalently be represented as by the following identity as well:
 \begin{eqnarray}
 Var(\theta|x)
 &=& E\big[(1-\kappa)|x,\tau\big]+ x^2 E\big[(1-\kappa)^2|x,\tau\big] - x^2E^2\big[1-\kappa|x,\tau\big]\nonumber\\
 &=& \frac{T_{\tau}(x)}{x} - T_{\tau}^2(x) + x^2 \frac{\int_{0}^{\infty}\frac{(t\tau^2)^2}{(1+t\tau^2)^{5/2}}t^{-a-1}L(t)e^{-\frac{x^2}{2(1+t\tau^2)}}dt}{\int_{0}^{\infty}\frac{1}{(1+t\tau^2)^{1/2}}t^{-a-1}L(t)e^{-\frac{x^2}{2(1+t\tau^2)}}dt}.\nonumber
 \end{eqnarray}
 That $Var(\theta|x) \leqslant 1 + x^2$ now follows trivially from the above identities.
\end{proof}

\begin{lem}\label{LEMMA_IMPORTANT_INTEGRAL}
 Let us define
 \begin{equation}
  J(x,\tau)=x^2 \frac{\int_{0}^{\infty}\frac{(t\tau^2)^2}{(1+t\tau^2)^{5/2}}t^{-a-1}L(t)e^{-\frac{x^2}{2(1+t\tau^2)}}dt}{\int_{0}^{\infty}\frac{1}{(1+t\tau^2)^{1/2}}t^{-a-1}L(t)e^{-\frac{x^2}{2(1+t\tau^2)}}dt}, x\in\mathbb{R} \mbox{ and } \tau\in(0,1),\nonumber
 \end{equation}
 where the function $L(\cdot)$ is already defined in (\ref{LAMBDA_PRIOR}) and satisfies Assumption \ref{ASSUMPTION_LAMBDA_PRIOR}, with $a\in[0.5,1)$. Then, for each $x\in\mathbb{R}$ and every $0<\tau<1$, the function $J(\cdot,\cdot)$ is bounded above by,
 \begin{equation}
  J(x,\tau) \leqslant 2KMe^{\frac{x^2}{2}}\tau^{2a}\big(1+o(1)),\label{CRUCIAL_UB_VAR}
 \end{equation}
 where the $o(1)$ term is independent of $x$, and depends only on $\tau$ such that the term $(1+o(1))$ in (\ref{CRUCIAL_UB_VAR}) is positive for any $0<\tau<1$, and $\lim_{\tau\rightarrow 0}o(1)=0$.
\end{lem}

\begin{proof}
First observe that, for each fixed $\tau\in (0,1)$, the function $J(x,\tau)$ is symmetric in $x$, and so is the stated upper bound in (\ref{CRUCIAL_UB_VAR}). Moreover, $J(0,\tau)=0$, for any $0<\tau<1$. Thus, the stated result is vacuously true when $x=0$ and $0<\tau<1$. Therefore, it will be enough to prove that (\ref{CRUCIAL_UB_VAR}) holds when $x>0$. So, let us assume that $x>0$.\newline

Note that
 \begin{eqnarray}
J(x,\tau)
% &=& x^2 \frac{\int_{0}^{\infty}\frac{(t\tau^2)^2}{(1+t\tau^2)^{5/2}}t^{-a-1}L(t)e^{-\frac{x^2}{2(1+t\tau^2)}}dt}{\int_{0}^{\infty}\frac{1}{(1+t\tau^2)^{1/2}}t^{-a-1}L(t)e^{-\frac{x^2}{2(1+t\tau^2)}}dt}.\nonumber\\
&\leqslant& Mx^2e^{\frac{x^2}{2}} \frac{\int_{0}^{\infty}\frac{(t\tau^2)^2}{(1+t\tau^2)^{5/2}}t^{-a-1}e^{-\frac{x^2}{2(1+t\tau^2)}}dt}{\int_{0}^{\infty}\frac{1}{(1+t\tau^2)^{1/2}}t^{-a-1}L(t)dt}\nonumber\\
&=& KMx^2e^{\frac{x^2}{2}} \int_{0}^{\infty}\frac{(t\tau^2)^2}{(1+t\tau^2)^{5/2}}t^{-a-1}e^{-\frac{x^2}{2(1+t\tau^2)}}dt \big(1+o(1)\big)\label{VAR_NONZERO_2}
\end{eqnarray}
where in the preceding chain of inequalities we use the facts that the function $L(\cdot)$ is bounded above by the constant $M > 0$ and $\int_{0}^{\infty}\frac{1}{(1+t\tau^2)^{1/2}}t^{-a-1}L(t)dt=\int_{0}^{\infty}t^{-a-1}L(t)dt\big(1+o(1)\big)= K^{-1}\big(1+o(1)\big)$ as $\tau \rightarrow 0$ (which follows from Lebesgue's Dominated Convergence Theorem). Clearly, the $o(1)$ term does not involve $x$ and depends only on $\tau$ such that $\lim_{\tau\rightarrow 0}o(1)=0$. It is also evident that the term $(1+o(1))$ in (\ref{CRUCIAL_UB_VAR}) is always positive for any $0<\tau<1$.\newline
% , we have
% \begin{eqnarray}
% \widetilde{J}(x,\tau)
% &\leqslant& Mx^2e^{\frac{x^2}{2}} \frac{\int_{0}^{\infty}\frac{(t\tau^2)^2}{(1+t\tau^2)^{5/2}}t^{-a-1}e^{-\frac{x^2}{2(1+t\tau^2)}}dt}{\int_{0}^{\infty}\frac{1}{(1+t\tau^2)^{1/2}}t^{-a-1}L(t)dt}\nonumber\\
% &=& KMx^2e^{\frac{x^2}{2}} \int_{0}^{\infty}\frac{(t\tau^2)^2}{(1+t\tau^2)^{5/2}}t^{-a-1}e^{-\frac{x^2}{2(1+t\tau^2)}}dt \big(1+o(1)\big)\nonumber
% \end{eqnarray}

Consider now the following variable transformation in the integral in (\ref{VAR_NONZERO_2}):
\begin{equation}
 u=\frac{x^2}{1+t\tau^2}.\nonumber
\end{equation}

Then we have,
\begin{eqnarray}
J(x,\tau)
&\leqslant&  KMx^2e^{\frac{x^2}{2}} \int_{0}^{x^2} \bigg(1-\frac{u}{x^2}\bigg)^2\bigg(\frac{u}{x^2}\bigg)^{1/2}\bigg(\frac{1}{\tau^2}\frac{x^2}{u}\big(1-\frac{u}{x^2}\big)\bigg)^{-a-1}e^{-u/2}\frac{x^2}{\tau^2u^2}du\big(1+o(1)\big)\nonumber\\
&=&  KM(x^2)^{1/2-a}e^{\frac{x^2}{2}}\tau^{2a}\int_{0}^{x^2} \bigg(1-\frac{u}{x^2}\bigg)^{1-a}u^{a-1/2}e^{-u/2}du\big(1+o(1)\big)\label{VAR_NONZERO_3}
\end{eqnarray}

Note that $0 < u < x^2$ implies $0 < 1-\frac{u}{x^2} < 1$. Hence, $\big(1-\frac{u}{x^2}\big)^{1-a}<1$ as $0<a<1$. Also, as $\frac{1}{2}\leqslant a <1$, we have, $u^{a-1/2}\leqslant (x^2)^{a-1/2}$.\newline

Therefore using (\ref{VAR_NONZERO_3}), we obtain,
\begin{eqnarray}\label{VAR_NONZERO_4}
J(x,\tau)
&\leqslant&  KMe^{\frac{x^2}{2}}\tau^{2a}\int_{0}^{x^2} e^{-u/2}du\big(1+o(1)\big),\nonumber\\
&=& 2KMe^{\frac{x^2}{2}}\tau^{2a} \big(1-e^{-x^2/2}\big)\big(1+o(1)\big)\nonumber\\
&\leqslant& 2KMe^{\frac{x^2}{2}}\tau^{2a}\big(1+o(1)\big),\nonumber
% &\leqslant& 2KM\big(1+o(1)\big) \quad[\mbox{since } |x| \leqslant \zeta_{\tau}] 
\end{eqnarray}
thereby completing the proof of Lemma \ref{LEMMA_IMPORTANT_INTEGRAL}.
\end{proof}

\paragraph{Proof of Theorem \ref{THM_3.2}}
\begin{proof}
 Suppose that $X \sim \mathcal{N}_{n}(\theta,I_{n})$, where $\theta \in l_{0}[q_{n}]$ and $\tilde{q}_n=\#\{i:\theta_{i}\neq 0\}$. Then $\tilde{q}_n \leqslant q_n$. Let $\zeta_{\tau}=\sqrt{2\log\big(\frac{1}{\tau^{2a}}\big)}$.\newline
 
 {\it Nonzero means:}\newline

 By applying the same reasoning as in the proof of Lemma \ref{LEMMA_MOST_IMP} to the final term of $Var(\theta|x)$ in (\ref{VAR_IDNT_1}), there exists a non-negative real-valued measurable function $\tilde{h}(x,\tau)$ such that $Var(\theta|x) \leqslant \tilde{h}(x,\tau)$ for all $x\in\mathbb{R}$ and all $\tau \in (0,1)$, where $\tilde{h}(x,\tau) \rightarrow 1 $ as $x \rightarrow\infty$ for any fixed $\tau \in (0,1)$. If $\tau \rightarrow 0$, the function $\tilde{h}(x,\tau)$ satisfies the following for any $c > 1$:
 \begin{eqnarray}
  \lim_{\tau \downarrow 0}\sup_{|x| > \rho \zeta_{\tau}} \tilde{h}(x,\tau) = 1 \mbox{ for all } \rho > c. \nonumber
 \end{eqnarray}

So, let us fix any arbitrary $c>1$ and choose any $\rho>c$. Then using the above arguments it follows that,
\begin{equation}\label{NONZERO_THETA_VAR_CONTRIBUTION_COMPONENT_1}
 E_{\theta_{i}}\big[Var(\theta_{i}|X_i)1\{|X_i|> \rho\zeta_{\tau}\}\big]
\lesssim 1 \mbox{ as } \tau \rightarrow 0.
\end{equation}

Let us now consider the case $|x| \leqslant \rho\zeta_{\tau}$. Then using the fact that $Var(\theta|x)\leqslant 1+x^2$ for any $x\in\mathbb{R}$ as obtained from Lemma \ref{POST_VAR_IDENTITY}, we obtain,
\begin{equation}\label{NONZERO_THETA_VAR_CONTRIBUTION_COMPONENT_2}
 E_{\theta_{i}}\big[Var(\theta_{i}|X_i)1\{|X_i|\leqslant \rho\zeta_{\tau}\}\big]
\lesssim \zeta_{\tau}^2 \mbox{ as } \tau \rightarrow 0.
\end{equation}

Combining (\ref{NONZERO_THETA_VAR_CONTRIBUTION_COMPONENT_1}) and (\ref{NONZERO_THETA_VAR_CONTRIBUTION_COMPONENT_2}) together, it follows that, for any $i$ such that $\theta_i\neq 0$,
% for any $i$ such that $\theta_i\neq 0$, we have, \label{NONZERO_THETA_VAR_CONTRIBUTION}
\begin{eqnarray}
\begin{aligned}
&\quad E_{\theta_{i}} Var(\theta_{i}|X_i)\nonumber\\
&=E_{\theta_{i}}\big[Var(\theta_{i}|X_i)1\{|X_i|> \rho\zeta_{\tau}\}\big]+E_{\theta_{i}}\big[Var(\theta_{i}|X_i)1\{|X_i|\leqslant \rho\zeta_{\tau}\}\big]\nonumber\\
&\lesssim 1+\zeta_{\tau}^2 \mbox{ as } \tau \rightarrow 0,
\end{aligned}
\end{eqnarray}
which holds uniformly in $i$ such that $\theta_i\neq 0$. Thus,
\begin{equation}
 \sum_{i:\theta_{i}\neq 0}E_{\theta_{i}} Var(\theta_{i}|X_i) \lesssim \tilde{q}_n(1+\zeta_{\tau}^2) \mbox{ as } \tau \rightarrow 0.\label{VAR_NONZERO_CONTRIBUTION}
\end{equation}

{\it Zero means:}\newline

 When $|x|>\zeta_{\tau}$, using the bound $Var(\theta|x) \leqslant 1+x^2$ in Lemma \ref{POST_VAR_IDENTITY}, we obtain,
 \begin{eqnarray}
  E_{0} Var(\theta|X_i)1\{|X_i| > \zeta_{\tau}\}
  &\leqslant& 2 \int_{\zeta_{\tau}}^{\infty} (1+x^2)\frac{1}{\sqrt{2\pi}}e^{-\frac{x^2}{2}}dx\nonumber\\
  &\lesssim& \frac{\tau^{2a}}{\zeta_{\tau}} + \zeta_{\tau}\tau^{2a} \mbox{ as } \tau \rightarrow 0.\label{VAR_THETA_0_1}
 \end{eqnarray}

 Again when $|x| \leqslant \zeta_{\tau}$, we consider the upper bound $Var(\theta|x) \leqslant \frac{T_{\tau}(x)}{x} + J(x,\tau)$ as obtained from Lemma \ref{POST_VAR_IDENTITY}, where the term $J(x,\tau)$ is already defined in Lemma \ref{LEMMA_IMPORTANT_INTEGRAL}. Also note that $\frac{T_{\tau}(x)}{x}=E(1-\kappa|x,\tau)$. Therefore, using the moment inequality given in Lemma \ref{LEM_MOMENT_INEQ}, it follows
\begin{equation}
\int_{-\zeta_{\tau}}^{\zeta_{\tau}} \frac{T_{\tau}(x)}{x}\frac{1}{\sqrt{2\pi}}e^{-\frac{x^2}{2}}dx 
\lesssim \tau^{2a}\zeta_{\tau} \mbox{ as } \tau \rightarrow 0.\label{VAR_THETA_0_2}
\end{equation}

Similarly, using Lemma \ref{LEMMA_IMPORTANT_INTEGRAL} we obtain,
\begin{equation}
\int_{-\zeta_{\tau}}^{\zeta_{\tau}} J(x,\tau)\frac{1}{\sqrt{2\pi}}e^{-\frac{x^2}{2}}dx 
\lesssim \tau^{2a}\zeta_{\tau} \mbox{ as } \tau \rightarrow 0.\label{VAR_THETA_0_3}
\end{equation}

On combining (\ref{VAR_THETA_0_2}) and (\ref{VAR_THETA_0_3}), it follows that
% for any $i$ such that $\theta_i=0$,
\begin{equation}\label{VAR_THETA_0_4}
E_{0}\big[Var(\theta_{i}|X_i)1\{|X_i|\leqslant \zeta_{\tau}\}\big]
\lesssim \tau^{2a}\zeta_{\tau} \mbox{ as } \tau \rightarrow 0.
\end{equation}

Therefore, using (\ref{VAR_THETA_0_1}) and (\ref{VAR_THETA_0_4}), it follows that for any $i$ such that $\theta_i=0$,
\begin{eqnarray}
\begin{aligned}
&\quad E_{0} Var(\theta_{i}|X_i)\nonumber\\
&=E_{0}\big[Var(\theta_{i}|X_i)1\{|X_i|> \zeta_{\tau}\}\big]+E_{0}\big[Var(\theta_{i}|X_i)1\{|X_i|\leqslant \zeta_{\tau}\}\big]\nonumber\\
&\lesssim \tau^{2a}\zeta_{\tau} \mbox{ as } \tau \rightarrow 0,
\end{aligned}
\end{eqnarray}
which holds uniformly in $i$ such that $\theta_i=0$. Consequently,
% Therefore, using (\ref{ZERO_THETA_VAR_CONTRIBUTION}) above we obtain, \label{ZERO_THETA_VAR_CONTRIBUTION}
\begin{equation}
 \sum_{i:\theta_{i}=0} E_{0} Var(\theta_i|X_i) \lesssim (n-\tilde{q}_n)\tau^{2a}\zeta_{\tau} \mbox{ as } \tau \rightarrow 0.\label{VAR_ZERO_CONTRIBUTION} 
\end{equation}
Combining (\ref{VAR_NONZERO_CONTRIBUTION}) and (\ref{VAR_ZERO_CONTRIBUTION}), it follows that
\begin{equation}
E_{\theta} \sum_{i=1}^{n} Var(\theta_{i}|X_i) \lesssim
  \tilde{q}_n\log\big(\frac{1}{\tau^{2a}}\big)+(n-\tilde{q}_n)\tau^{2a}\sqrt{\log\big(\frac{1}{\tau^{2a}}\big)} \mbox{ as } \tau \rightarrow 0.\nonumber
 \end{equation}
The result then follows immediately by noting that $\tilde{q}_n\leqslant q_n$ and $q_n=o(n)$ and subsequently taking supremum over all $\theta\in l_{0}[q_{n}]$. This completes the proof of Theorem \ref{THM_3.2}.
\end{proof}
%  and subsequently taking supremum over all $\theta_{0}\in l_{0}[q_{n}]$.
\begin{lem}\label{LEMMA_POST_VAR_LOWER_BOUND}
 Suppose the function $L(\cdot)$ given by (\ref{LAMBDA_PRIOR}) satisfies Assumption \ref{ASSUMPTION_LAMBDA_PRIOR} and is non-decreasing over $(0,\infty)$, with $a=\frac{1}{2}$. Let us define for fixed $y > 0$ and for fixed $k > 0$,
 \begin{equation}
 I_{k}= \int_{0}^{\infty} \frac{(t\tau^2)^{k-\frac{1}{2}}}{(1+t\tau^2)^{k}}t^{-3/2}L(t)e^{\frac{t\tau^2}{1+t\tau^2}y}dt. \nonumber
 \end{equation}
 Then,
 \begin{align*}
  I_{\frac{5}{2}} &\geqslant L(1)\tau\bigg[\frac{\tau}{y}\big(e^{y/2} - e^{\tau^2y}\big)+\frac{1}{\sqrt{2}y}\big(e^{y}-e^{y/2}\big) \bigg], \mbox{ for } \tau < \frac{1}{\sqrt{2}}\nonumber\\
  I_{\frac{1}{2}}  &\leqslant \tau\bigg[\frac{e^{\tau^2y}}{K\tau}+2Me^{\tau y}\big(\frac{1}{\tau}-\frac{1}{\sqrt{\tau}}\big)+2Me^{\frac{y}{2}}\big(\frac{1}{\sqrt{\tau}} - \sqrt{2} \big)+ \frac{2M\sqrt{2}}{y}\big(e^{y} - e^{\frac{y}{2}}\big)\bigg], \mbox{ for } \tau < \frac{1}{2}\nonumber\\
  I_{\frac{3}{2}} &\leqslant M\tau\bigg[e^{\tau^2y}\tau +2e^{\frac{y}{2}}\big(\frac{1}{\sqrt{2}}-\tau \big)+\frac{\sqrt{2}}{y}\big(e^{y}-e^{\frac{y}{2}}\big)\bigg], \mbox{ for } \tau < \frac{1}{\sqrt{2}}\nonumber\\
  I_{\frac{1}{2}} &\geqslant L(1)\tau\bigg[e^{\tau^2y}\big(\frac{1}{\tau}-\frac{1}{\sqrt{\tau}}\big)+\frac{\sqrt{2}}{y}\big(e^{y}-e^{\tau y}\big)+\frac{1}{2y}\big(e^{y}-e^{\frac{y}{2}}\big)\bigg], \mbox{ for } \tau < \frac{1}{2}.\nonumber
 \end{align*}
 \end{lem}
 \begin{proof}
 Note that since $L$ is nondecreasing over $(0,\infty)$ and $0<\tau^2<1$, for each fixed $\tau \in (0,1)$ we have, $L(t) \geqslant L(1)$ for all $t\geqslant \frac{1}{1-\tau^2}> 1$. Therefore,
 \begin{eqnarray}
  I_{\frac{5}{2}}
  &=& \int_{0}^{\infty} \frac{(t\tau^2)^2}{(1+t\tau^2)^{5/2}}t^{-3/2}L(t)e^{\frac{t\tau^2}{1+t\tau^2}y}dt\nonumber\\
  &=&\frac{1}{\tau}\int_{0}^{\infty}\big(\frac{t\tau^2}{1+t\tau^2×}\big)^{5/2}t^{-2}L(t)e^{\frac{t\tau^2}{1+t\tau^2}y}dt\nonumber\\
  &\geqslant& \frac{L(1)}{\tau}\int_{\frac{1}{1-\tau^2}}^{\infty}\big(\frac{t\tau^2}{1+t\tau^2×}\big)^{5/2}t^{-2}e^{\frac{t\tau^2}{1+t\tau^2}y}dt\label{INEQ_1}
 \end{eqnarray}
 Now putting $u=\frac{t\tau^2}{1+t\tau^2}$ in (\ref{INEQ_1}) we obtain,
 \begin{align*}
  I_{\frac{5}{2}}
%   &\geqslant L(1)\tau \int_{\frac{\tau^2}{1+\tau^2}}^{1}u^{1/2}e^{uy}du\nonumber\\
  &\geqslant L(1)\tau \int_{\tau^2}^{1}u^{1/2}e^{uy}du\nonumber\\
  &=L(1)\tau\bigg[\frac{\tau}{y}\big(e^{y/2} - e^{\tau^2y}\big)+\frac{1}{\sqrt{2}y}\big(e^{y}-e^{y/2}\big) \bigg], \mbox{ for } \tau < \frac{1}{\sqrt{2}}\nonumber
 \end{align*}
where the last equality follows using the proof of Lemma A.1 of \citet{PKV2014}.\newline
  
Next observe that
\begin{eqnarray}
 I_{\frac{1}{2}}
 &=& \int_{0}^{\infty} \frac{1}{(1+t\tau^2)^{1/2}}t^{-3/2}L(t)e^{\frac{t\tau^2}{1+t\tau^2}y}dt\nonumber\\
 &=& \frac{1}{\tau}\int_{0}^{\infty}\big(\frac{t\tau^2}{1+t\tau^2}\big)^{1/2}t^{-2}L(t)e^{\frac{t\tau^2}{1+t\tau^2}y}dt\nonumber\\
 &=& \tau \int_{0}^{1}u^{-3/2}L\big(\frac{1}{\tau^2}\frac{u}{1-u}\big)e^{uy}du\mbox{ }[\mbox{putting }u=\frac{t\tau^2}{1+t\tau^2}] \nonumber\\
 &=& \tau \bigg[\int_{0}^{\tau^2}u^{-3/2}L\big(\frac{1}{\tau^2}\frac{u}{1-u}\big)e^{uy}du+\int_{\tau^2}^{1}u^{-3/2}L\big(\frac{1}{\tau^2}\frac{u}{1-u}\big)e^{uy}du\bigg]\label{INEQ_2}
\end{eqnarray}
Now observe that $e^{uy} \leqslant e^{\tau^2 y}$ for all $u \leqslant \tau^2$. Using this fact and applying the change of variable $t=\frac{1}{\tau^2}\frac{u}{1-u}$ in the first integral on the right hand side of (\ref{INEQ_2}) we obtain,
\begin{eqnarray}
 \int_{0}^{\tau^2}u^{-3/2}L\big(\frac{1}{\tau^2}\frac{u}{1-u}\big)e^{uy}du
 &\leqslant& e^{\tau^2 y} \int_{0}^{\tau^2}u^{-3/2}L\big(\frac{1}{\tau^2}\frac{u}{1-u}\big)du\nonumber\\
 &=& \frac{e^{\tau^2 y}}{\tau} \int_{0}^{\frac{1}{1-\tau^2}}\frac{t^{-3/2}}{\sqrt{1+t\tau^2}}L(t)dt\nonumber\\
%  &\leqslant& \frac{e^{\tau^2 y}}{\tau} \int_{0}^{\infty}\frac{t^{-3/2}}{\sqrt{1+t\tau^2}}L(t)dt\nonumber\\
 &\leqslant& \frac{e^{\tau^2 y}}{\tau} \int_{0}^{\infty}t^{-3/2}L(t)dt \mbox{ } [\mbox{since }\frac{1}{\sqrt{1+t\tau^2}} \leqslant 1]\nonumber\\
 &=& \frac{K^{-1}e^{\tau^2 y}}{\tau}\label{INEQ_3}
\end{eqnarray}
For the second integral on the right hand side of (\ref{INEQ_2}), we observe that the function $L$ is bounded by the constant $M > 0$. Using this observation and then apply the same arguments given in the proof of Lemma A.1 of \citet{PKV2014}, we obtain,
\begin{equation}
 \int_{\tau^2}^{1}u^{-3/2}L\big(\frac{1}{\tau^2}\frac{u}{1-u}\big)e^{uy}du
\leqslant 2M \bigg[e^{\tau y}\big(\frac{1}{\tau}-\frac{1}{\sqrt{\tau}}\big)+ e^{\frac{y}{2}}\big(\frac{1}{\sqrt{\tau}} - \sqrt{2}\big) + \frac{\sqrt{2}}{y}\big(e^{y} - e^{\frac{y}{2}}\big)\bigg] \mbox{} \mbox{ for } \tau < \frac{1}{2}\label{INEQ_4}.
\end{equation}
(\ref{INEQ_2}), (\ref{INEQ_3}) and (\ref{INEQ_4}) together immediately give the stated upper bound on $I_{\frac{1}{2}}$.\newline

Again note that since $ \tau^2 < u < 1$ we have $\frac{1}{\tau^2}\frac{u}{1-u} > \frac{1}{1-\tau^2}>1$. Hence $L\big(\frac{1}{\tau^2}\frac{u}{1-u}\big) \geqslant L\big(\frac{1}{1-\tau^2}\big) \geqslant L(1) $ as $L$ is nondecreasing. Using this observation and (\ref{INEQ_2}) and then applying the same reasoning given in the proof of Lemma A.1 of \citet{PKV2014}, we obtain,
\begin{eqnarray}
 I_{\frac{1}{2}}
 &\geqslant& \tau \int_{\tau^2}^{1}u^{-3/2}L\big(\frac{1}{\tau^2}\frac{u}{1-u}\big)e^{uy}du\nonumber\\
 &\geqslant& L(1)\tau \int_{\tau^2}^{1}u^{-3/2}e^{uy}du\nonumber\\
 &=& L(1)\tau\bigg[e^{\tau^2y}\big(\frac{1}{\tau}-\frac{1}{\sqrt{\tau}}\big)+\frac{\sqrt{2}}{y}\big(e^{y}-e^{\tau y}\big)+\frac{1}{2y}\big(e^{y}-e^{\frac{y}{2}}\big)\bigg], \mbox{ for } \tau < \frac{1}{2}.\nonumber
\end{eqnarray}
The stated upper bound for $I_{\frac{3}{2}}$ follows immediately by noting that $L$ is bounded by the constant $M > 0$ and subsequently by change the variable $u=t\tau^2/(1+t\tau^2)$ followed by the arguments used in the proof of Lemma A.1 of \citet{PKV2014}, thereby completing the proof of Lemma \ref{LEMMA_POST_VAR_LOWER_BOUND}.
 \end{proof}

 \paragraph{Proof of Theorem \ref{THM_3.4}}
 \begin{proof}
  From (\ref{VAR_IDNT_2}) we have,
 \begin{eqnarray}
   Var(\theta|x)
  &\geqslant& x^2E\bigg[(1-\kappa)^2 | x,\tau\bigg] - x^2E^2\bigg[(1-\kappa) | x,\tau\bigg] \nonumber\\
  &=& x^2 \bigg[\frac{\int_{0}^{\infty}\frac{(t\tau^2)^2}{(1+t\tau^2)^{5/2}}t^{-3/2}L(t)e^{-\frac{x^2}{2(1+t\tau^2)}}dt}{\int_{0}^{\infty}\frac{1}{(1+t\tau^2)^{1/2}}t^{-3/2}L(t)e^{-\frac{x^2}{2(1+t\tau^2)}}dt}
  - \bigg(\frac{\int_{0}^{\infty}\frac{t\tau^2}{(1+t\tau^2)^{3/2}}t^{-3/2}L(t)e^{-\frac{x^2}{2(1+t\tau^2)}}dt}{\int_{0}^{\infty}\frac{1}{(1+t\tau^2)^{1/2}}t^{-3/2}L(t)e^{-\frac{x^2}{2(1+t\tau^2)}}dt}\bigg)^2 \bigg]\nonumber\\
%   &=& 2y \bigg[\frac{\int_{0}^{\infty}(t\tau^2)^2\frac{e^{\frac{t\tau^2}{2(1+t\tau^2)}y}}{(1+t\tau^2)^{5/2}}t^{-3/2}L(t)dt}{\int_{0}^{\infty}\frac{e^{\frac{t\tau^2}{2(1+t\tau^2)}y}}{(1+t\tau^2)^{1/2}}t^{-3/2}L(t)dt}  - \bigg(\frac{\int_{0}^{\infty}t\tau^2\frac{e^{\frac{t\tau^2}{2(1+t\tau^2)}y}}{(1+t\tau^2)^{3/2}}t^{-3/2}L(t)dt}{\int_{0}^{\infty}\frac{e^{\frac{t\tau^2}{2(1+t\tau^2)}y}}{(1+t\tau^2)^{1/2}}t^{-3/2}L(t)dt}\bigg)^2 \bigg]\nonumber
&=& 2y \bigg[\frac{I_{\frac{5}{2}}}{I_{\frac{1}{2}}} - \bigg(\frac{I_{\frac{3}{2}}}{I_{\frac{1}{2}}}\bigg)^2 \bigg]\nonumber
\end{eqnarray}
 by putting $y=\frac{x^2}{2}$ and rest of the proof follows by applying Lemma \ref{LEMMA_POST_VAR_LOWER_BOUND} and the same set of arguments given in the proof of Theorem 3.4 of \citet{PKV2014}.  
 \end{proof}
 
 \begin{lem}\label{MOST_IMP_LEMMA_ABOS}
Let us fix any $0 <\eta<1$ and any $0<\delta<1$. Consider the general class of shrinkage priors where the prior distribution for the local shrinkage parameters $\pi(\lambda_i^2)$ is given by (\ref{LAMBDA_PRIOR}) with $a\in (0,1)$ and the corresponding slowly varying component $L(\cdot)$ satisfies Assumption \ref{ASSUMPTION_LAMBDA_PRIOR}. Then for each fixed $x\in\mathbb{R}$ and every fixed $0<\tau<1$, the corresponding posterior shrinkage coefficients $E(\kappa|x,\tau)$  can be bounded above by a real valued function $g(x,\tau)$, depending on $\eta$ and $\delta$, and is given by,
  \begin{equation}\label{g_definition}
  \mathop{g(x,\tau)=} \left\{
  \begin{array}{ll}
 C_{\textasteriskcentered} \big[\mid x^2 \int_{0}^{s x^2}e^{-u/2}u^{a+1/2-1}du \mid\big]^{-1} + \frac{H(a,\eta,\delta)e^{-\frac{\eta(1-\delta) x^2}{2}}}{\tau^{2a}\Delta(\tau^2,\eta,\delta)}, & \mbox{if}\; |x| > 0 ,\\ \\
  1, & \mbox{if}\; x=0,
  \end{array}\right.
  \end{equation}
for some global constant $C_{\textasteriskcentered} \equiv C_{\textasteriskcentered}(a,\eta,L) > 0$ which is independent of both $x$ and $\tau$. The function $g(x,\tau)$ defined in (\ref{g_definition}) above satisfies the following:\newline
 
 Given any $\zeta > \frac{2}{\eta(1-\delta)}$,
\begin{eqnarray}
 \lim_{\tau \downarrow 0}\sup_{|x| > \sqrt{\zeta \log(\frac{1}{\tau^{2a}})}} g(x,\tau) = 0. \nonumber
 \end{eqnarray}
\end{lem}
\begin{proof}
 Proof follows using the same set of arguments as in the proof of Lemma \ref{LEMMA_MOST_IMP}.
\end{proof}
 
 \begin{remark}
 For each fixed $\tau$, the function $g(x,\tau)$ defined in (\ref{g_definition}) has exactly one point of discontinuity at the origin and hence it is measurable.
 \end{remark}

 \paragraph{Proof of Theorem \ref{THM_T1_UB_LB}}
 \begin{proof}
 First of all, it should be observed that, for each $i$, $X_i\sim N(0,1)$ under $H_{0i}$, and hence the corresponding type I error probability does not depend on $i$. Let $t_1$ denote the common value of $t_{1i}$'s, $i=1,\dots,n$. Hence, by definition,
 \begin{equation}\label{TYPE_I_DEFN}
  t_1=\Pr\big(E(1-\kappa_1|X_1,\tau) > 0.5\big|H_{01} \mbox{ is true}\big)
 \end{equation}

 Now, the upper bound for $t_1$ as given in the statement of Theorem \ref{THM_T1_UB_LB}, is a simple consequence of Theorem 6 of \citet{GTGC2015} with some appropriately chosen finite positive constant $H(a,\eta,\delta)$ which is independent of both $i$ and $n$. Hence we omit the proof. However, the corresponding proof for the lower bound of $t_1$ require some novel arguments as given below.\newline
 
 Let us fix any $\eta\in(0,1)$ and any $\delta\in(0,1)$. Then, by Lemma \ref{MOST_IMP_LEMMA_ABOS}, for each fixed $\tau < 1$, the function $E(\kappa|x,\tau)$ is bounded above by the function $g(x,\tau)$, where the function $g(\cdot,\tau)$ has already been defined in (\ref{g_definition}). Hence, for each fixed $\tau < 1$ and for every $\omega$ in the sample space, we have the following:
 \begin{eqnarray}\label{TYPE_I_ERROR_LB_EQ1}
 \bigg\{\omega:E(1-\kappa_1|X_1(\omega),\tau) > 0.5\bigg\}
 &=&\bigg\{\omega:E(\kappa_1|X_1(\omega),\tau) < 0.5\bigg\}\nonumber\\
 &\supseteq& \bigg\{\omega: g(X_1(\omega),\tau) < 0.5\bigg\} \equiv B_n^c, \mbox{ say,}
 \end{eqnarray}
 where $B_n\equiv \big\{\omega: g(X_1(\omega),\tau) \geqslant 0.5\big\}$.\newline
 
 Let us fix any $\zeta > \frac{2}{\eta(1-\delta)}$ and define the event $C_n$ as $ C_n\equiv\bigg\{\omega:|X_1(\omega)| > \sqrt{\zeta\log(\frac{1}{\tau^{2a}})}\bigg\}$.\newline
 
 Then using (\ref{TYPE_I_DEFN}) and (\ref{TYPE_I_ERROR_LB_EQ1}), we have,
 \begin{eqnarray}\label{TYPE_I_ERROR_LB_EQ2}
  t_1
%   &=& \Pr(E(1-\kappa_i|X_i,\tau) > \frac{1}{2}|H_{01} \mbox{ is true})\nonumber\\
  &=& \Pr(E(\kappa_1|X_1,\tau) < 0.5|H_{01} \mbox{ is true})\nonumber\\
  &\geqslant& \Pr(g(X_1,\tau)< 0.5|H_{01} \mbox{ is true})\nonumber\\
  &=& \Pr(B_n^c|H_{01} \mbox{ is true})\nonumber\\
  &\geqslant& \Pr(B_n^c \cap C_n | H_{01} \mbox{ is true})\nonumber\\
  &=& \Pr(B_n^c|C_n, H_{01} \mbox{ is true})\Pr(C_n|H_{01} \mbox{ is true})\nonumber\\
  &=& \big[1- \Pr(B_n|C_n, H_{01} \mbox{ is true})\big]\Pr(C_n|H_{01} \mbox{ is true}).
 \end{eqnarray}

  Recall that the function $g(x,\tau)$ is measurable, non-negative and is continuously decreasing in $|x|$ for $|x|\neq0$. Hence $E\big(g(X_1,\tau)||X_1| > \sqrt{\zeta\log(1/\tau^{2a})},H_{01} \mbox{ is true}\big)$ is well defined and is bounded for all sufficiently small $\tau\in(0,1)$. Using these facts and applying Markov's inequality, we obtain for all sufficiently $\tau<1$, the following:
 \begin{eqnarray}\label{TYPE_I_ERROR_LB_EQ3}
   \Pr(B_n\big|C_n, H_{01} \mbox{ is true})
  &=& \Pr\big(g(X_1,\tau) \geqslant 0.5\big||X_1|>\sqrt{\zeta\log(1/\tau^{2a})}, H_{01} \mbox{ is true}\big)\nonumber\\
  &\leqslant& 2E\big(g(X_1,\tau)\big||X_1|>\sqrt{\zeta\log(1/\tau^{2a})}, H_{01} \mbox{ is true}\big)\nonumber\\
  &\leqslant& 2\sup_{\mid x\mid>\sqrt{\zeta \log(\frac{1}{\tau^{2a}})}} g(x,\tau)\nonumber\\
  &\rightarrow& 0 \mbox{ as } \tau \rightarrow 0.\nonumber
     \end{eqnarray}
     
 Since $\tau\rightarrow 0$ as $n\rightarrow\infty$, we have,
 \begin{equation}
 \lim\limits_{n\rightarrow\infty} \Pr(B_n\big|C_n, H_{01} \mbox{ is true})=0,\nonumber
 \end{equation}
  whence it follows
  \begin{equation}
  \label{TYPE_I_ERROR_LB_EQ4}
  \Pr(B_n^c|C_n, H_{01} \mbox{ is true})= 1-o(1) \mbox{ as } n\rightarrow\infty.
 \end{equation}
 
 Again, noting that under $H_{01}$, $X_1\stackrel{d}{=}Z$, we have for all sufficiently small $\tau < 1$, the following:
 \begin{eqnarray}\label{TYPE_I_ERROR_LB_EQ5}
  \Pr(C_n|H_{01} \mbox{ is true})
%   &=& \Pr\big(|X_i| > \sqrt{\zeta\log(1/\tau^{2a})}\big|H_{01} \mbox{ is true}\big)\nonumber\\
  &=& \Pr\big(|Z| > \sqrt{\zeta\log(1/\tau^{2a})}\big)\nonumber\\
  &=& 2\Pr\big(Z > \sqrt{\zeta\log(1/\tau^{2a})}\big)\nonumber\\
  &\geqslant& 2\frac{\phi(\sqrt{\zeta\log(1/\tau^{2a})})}{\sqrt{\zeta\log(1/\tau^{2a})}}\big(1-\frac{1}{\zeta\log(1/\tau^{2a})}\big) [\mbox{using Mill's ratio}]\nonumber\\
  &\geqslant& 2\frac{\phi(\sqrt{\zeta\log(1/\tau^{2a})})}{\sqrt{\zeta\log(1/\tau^{2a})}}\big(1-\frac{1}{2\log(1/\tau^{2a})}\big) [\mbox{since } \zeta > 2 ]\nonumber\\
  &=& G(a,\eta,\delta)\frac{\big(\tau^{2a}\big)^{\frac{\zeta}{2}}}{\sqrt{\log(\frac{1}{\tau^2})}}(1+o(1))\mbox{ as } n\rightarrow\infty.
 \end{eqnarray}
 On combining (\ref{TYPE_I_ERROR_LB_EQ2}), (\ref{TYPE_I_ERROR_LB_EQ4}) and (\ref{TYPE_I_ERROR_LB_EQ5}), the stated result follows immediately. 
\end{proof}
 
 \paragraph{Proof of Theorem \ref{THM_T2_UB_LB}}
 \begin{proof}
  First observe that, for each $i$, $X_i \sim N(0,1+\psi^2)$ under $H_{Ai}$, and hence the corresponding type II error probability does not depend on $i$. Let $t_1$ denote the common value of $t_{2i}$'s, $i=1,\dots,n$. Hence, by definition,
 \begin{equation}\label{TYPE_I_DEFN}
  t_2=\Pr\big(E(1-\kappa_1|X_1,\tau) \leqslant 0.5\big|H_{A1} \mbox{ is true}\big)
 \end{equation}

 Now, the lower bound for $t_2$ as given in the statement of Theorem \ref{THM_T2_UB_LB}, is a simple consequence of Theorem 8 of \citet{GTGC2015}. Hence we skip the proof of this part. However, the corresponding proof for the upper bound of $t_2$ require some novel arguments as given below.\newline

 Let us fix any $\eta \in (0,1)$ and any $\delta \in (0,1)$. Choose any $\zeta > \frac{2}{\eta(1-\delta)}$. Then using Lemma \ref{MOST_IMP_LEMMA_ABOS}, it follows that, for each fixed $\tau < 1$, the function $E(\kappa|x,\tau)$ is bounded above by the function $g(x,\tau)$, where the function $g(\cdot,\tau)$ has already been defined in (\ref{g_definition}.) Hence, for each fixed $\tau < 1$, we have the following:
  \begin{eqnarray}\label{TYPE_II_UB_EQ1}
   t_2 &=& \Pr\big(E(\kappa_1|X_1,\tau) \geqslant 0.5\big|H_{A1} \mbox{ is true}\big)\nonumber\\
       &\leqslant& \Pr\big(g(X_1,\tau) \geqslant 0.5\big|H_{A1} \mbox{ is true}\big)\nonumber\\
       &=& \Pr(B_n\big|H_{A1} \mbox{ is true})\nonumber\\
       &=& \Pr(B_n \cap C_n\big|H_{A1} \mbox{ is true})+\Pr(B_n \cap C^c_n\big|H_{A1} \mbox{ is true})\nonumber\\
       &=& \Pr(B_n \big|C_n, H_{A1} \mbox{ is true})\Pr(C_n\big|H_{A1} \mbox{ is true})+\Pr(B_n \cap C^c_n\big|H_{A1} \mbox{ is true})\nonumber\\
       &\leqslant& \Pr(B_n\big|C_n, H_{A1} \mbox{ is true})+\Pr(C^c_n\big|H_{A1} \mbox{ is true}).
     \end{eqnarray}    
 where the events $B_n$ and $C_n$ are already defined in the proof of Theorem \ref{THM_T1_UB_LB}.\newline
     
%  Recall that the function $g(x,\tau)$ is measurable, non-negative and is continuously decreasing in $|x|$ when $|x|\neq0$. Hence $E\big(g(X_1,\tau)\big||X_1| > \sqrt{\zeta\log(1/\tau^{2a})},H_{A1} \mbox{ is true}\big)$ is well defined and is bounded for all sufficiently small $\tau\in(0,1)$. Using these facts and applying Markov's inequality, we obtain for all sufficiently $\tau<1$, the following:
%  \begin{eqnarray}
%    \Pr(B_n\big|C_n, H_{A1} \mbox{ is true})
%   &=& \Pr\bigg(g(X_1,\tau) \geqslant 0.5\bigg||X_1|>\sqrt{\zeta\log(1/\tau^{2a})}, H_{A1} \mbox{ is true}\bigg)\nonumber\\
%   &\leqslant& 2E\bigg(g(X_1,\tau)\bigg||X_1|>\sqrt{\zeta\log(1/\tau^{2a})}, H_{A1} \mbox{ is true}\bigg)\nonumber\\
%   &\leqslant& 2\sup_{\mid x\mid>\sqrt{\zeta \log(\frac{1}{\tau^{2a}})}} g(x,\tau)\nonumber\\
%   &\rightarrow& 0 \mbox{ as } \tau \rightarrow 0.\nonumber
%      \end{eqnarray}
     
Now applying the same reasoning as in the proof of Theorem \ref{THM_T1_UB_LB} it follows that 
 \begin{equation}
  \label{TYPE_II_UB_EQ2}
 \lim\limits_{n\rightarrow\infty} \Pr(B_n\big|C_n, H_{A1} \mbox{ is true})=0.
 \end{equation}
     
 Again, since $0<\lim_{n\rightarrow\infty}\frac{\tau}{p^{\alpha}}<\infty$, for $\alpha>0$, using condition (iii) of Assumption \ref{ASSUMPTION_ASYMP}, it follows that $\lim_{n\rightarrow\infty}\log(\frac{1}{\tau^2})/\psi^2=\alpha C$. Thus, for all sufficiently large $n$, we have the following:
  \begin{eqnarray}\label{TYPE_II_UB_EQ3}
   \Pr(C^c_n\big|H_{A1} \mbox{ is true})
   &=& \Pr\bigg(|X_1| \leqslant\sqrt{\zeta\log(\frac{1}{\tau^{2a}})} \big| H_{A1} \mbox{ is true}\bigg)\nonumber\\
   &=& \Pr\bigg(|Z| \leqslant\sqrt{\zeta a}\sqrt{\frac{\log(\frac{1}{\tau^2})}{1+\psi^2}}\bigg)\nonumber\\
   &=& \Pr\bigg(|Z| \leqslant\sqrt{\zeta a}\sqrt{\frac{\log(\frac{1}{\tau^2})}{\psi^2}}\big(1+o(1)\big)\bigg) \mbox{ as } n\rightarrow\infty\nonumber\\
   &=& \Pr\big(|Z| \leqslant\sqrt{\zeta a\alpha}\sqrt{C})(1+o(1))\big)\mbox{ as }n\rightarrow\infty \nonumber\\
   &=& \big[2\Phi(\sqrt{\zeta a\alpha}\sqrt{C})-1\big](1+o(1)) \mbox{ as }n\rightarrow\infty.
  \end{eqnarray}
 Combining (\ref{TYPE_II_UB_EQ1}), (\ref{TYPE_II_UB_EQ2}) and (\ref{TYPE_II_UB_EQ3}), it therefore follows that
 \begin{equation}
  t_2 \leqslant \big[2\Phi(\sqrt{\zeta a\alpha}\sqrt{C})-1\big](1+o(1)),\nonumber
 \end{equation}
 for all sufficiently large $n$, thereby completing the proof of Theorem \ref{THM_T2_UB_LB}.
  \end{proof}

\paragraph{Proof of Theorem \ref{THM_BAYES_RISK_EXACT_ASYMP_ORDER}}
\begin{proof}
Let us first fix any $\eta \in (0,1)$ and any $\delta \in (0,1)$. Then combining the results of Theorem \ref{THM_T1_UB_LB} and Theorem \ref{THM_T2_UB_LB}, together with (\ref{OPT_BAYES_RISK}) and the arguments employed for proving Theorem 1 of \citet{GTGC2015}, it follows easily that, that the Bayes risk $R_{OG}$ of the induced decisions, (\ref{INDUCED_DECISION}) based on the general class of one-group priors under study, with $a\in[0.5,1)$, satisfies the following:
  \begin{equation}\label{RATIO_RISK_OG_BOUNDS}
 \frac{2\Phi\big(\sqrt{2a\alpha}\sqrt{C}\big)-1}{2\Phi\big(\sqrt{C}\big)-1}\big(1+o(1)\big) \leqslant  \frac{R_{OG}}{R^{BO}_{Opt}} \leqslant \frac{2\Phi\big(\sqrt{\zeta a\alpha}\sqrt{C}\big)-1}{2\Phi\big(\sqrt{C}\big)-1}\big(1+o(1)\big)\mbox{ as } n\rightarrow\infty,
  \end{equation}
%    for all sufficiently large $n$
  for any arbitrary but fixed $\zeta > \frac{2}{\eta(1-\delta)}$, where the $o(1)$ terms depend on choices of $\zeta$, $\eta$ and $\delta$ only. Then taking limit inferior and limit superior in (\ref{RATIO_RISK_OG_BOUNDS}) as $n\rightarrow\infty$, it follows
  \begin{equation}\label{RATIO_RISK_LIMINF_LIMSUP}
  \frac{2\Phi\big(\sqrt{2a\alpha}\sqrt{C}\big)-1}{2\Phi\big(\sqrt{C}\big)-1}
  \leqslant \liminf_{n\rightarrow\infty}\frac{R_{OG}}{R_{Opt}^{BO}}
  \leqslant \limsup_{n\rightarrow\infty}\frac{R_{OG}}{R_{Opt}^{BO}}
  \leqslant \frac{2\Phi\big(\sqrt{\zeta a\alpha}\sqrt{C}\big)-1}{2\Phi\big(\sqrt{C}\big)-1}
 \end{equation}
 for each fixed $\eta \in (0,1)$ and each fixed $\delta \in (0,1)$ and for any $\zeta > 2/(\eta(1-\delta))$.\newline
 
 Now observe that the multiple testing rules under study do not depend on how $\eta\in (0,1)$, $\delta\in(0,1)$ and $\zeta > 2/(\eta(1-\delta))$ are chosen. Hence the ratio $R_{OG}/R_{Opt}^{BO}$ is free of any $\eta, \delta \in (0,1)$ and any $\zeta > 2/(\eta(1-\delta))$, for all $n\geqslant 1$. Thus, the limit inferior and the limit superior terms in (\ref{RATIO_RISK_LIMINF_LIMSUP}) are also independent of the choices of $\eta$, $\delta$ and $\zeta$. But $\zeta>2/(\eta(1-\delta))$ in (\ref{RATIO_RISK_LIMINF_LIMSUP}) is arbitrary. Therefore, taking infimum over all such $\zeta$'s in (\ref{RATIO_RISK_LIMINF_LIMSUP}) and using the continuity of $\Phi(\cdot)$, we obtain,
  \begin{eqnarray}
  \frac{2\Phi\big(\sqrt{2a\alpha}\sqrt{C}\big)-1}{2\Phi\big(\sqrt{C}\big)-1}
   \leqslant \liminf_{n\rightarrow\infty}\frac{R_{OG}}{R_{Opt}^{BO}}
   \leqslant \limsup_{n\rightarrow\infty}\frac{R_{OG}}{R_{Opt}^{BO}}
%    &\leqslant& \inf_{\zeta > \frac{2}{\eta(1-\delta)}} \frac{2\Phi\big(\sqrt{\zeta/2}\sqrt{C}\big)-1}{2\Phi\big(\sqrt{C}\big)-1}\nonumber\\
   \leqslant\frac{2\Phi\big(\sqrt{\frac{2 a\alpha}{\eta(1-\delta)}}\sqrt{C}\big)-1}{2\Phi\big(\sqrt{C}\big)-1}\label{RATIO_RISK_INF_ZETA}
  \end{eqnarray}
  
  Once again (\ref{RATIO_RISK_INF_ZETA}) holds for every fixed $\eta \in (0,1)$ and every fixed $\delta \in (0,1)$. Hence, using the preceding discussion and taking infimum in (\ref{RATIO_RISK_INF_ZETA}) over all possible choices of $(\eta, \delta) \in (0,1)\times(0,1)$ and using the continuity of $\Phi(\cdot)$, we finally obtain
  \begin{eqnarray}
  \frac{2\Phi\big(\sqrt{2a\alpha}\sqrt{C}\big)-1}{2\Phi\big(\sqrt{C}\big)-1}
   \leqslant \liminf_{n\rightarrow\infty}\frac{R_{OG}}{R_{Opt}^{BO}}
   \leqslant \limsup_{n\rightarrow\infty}\frac{R_{OG}}{R_{Opt}^{BO}}
%    \leqslant \inf_{(\eta,\delta) \in (0,1)\times(0,1)} \frac{2\Phi\big(\sqrt{\frac{C}{\eta(1-\delta)}}\big)-1}{2\Phi\big(\sqrt{C}\big)-1}\nonumber\\
   \leqslant\frac{2\Phi\big(\sqrt{2a\alpha}\sqrt{C}\big)-1}{2\Phi\big(\sqrt{C}\big)-1}\nonumber
  \end{eqnarray}
whence we have
 \begin{align}
%  \label{EXACT_ASYMP_ORDER_RATIO_RISKS}
   \lim_{n\rightarrow\infty}\frac{R_{OG}}{R_{Opt}^{BO}}=\frac{2\Phi\big(\sqrt{2a\alpha}\sqrt{C}\big)-1}{2\Phi\big(\sqrt{C}\big)-1}.\nonumber
 \end{align}
That $\lim\limits_{n\rightarrow\infty}\frac{R_{OG}}{R_{Opt}^{BO}}=1$ for $a=0.5$ and $\alpha=1$ is obvious. This completes the proof of Theorem \ref{THM_BAYES_RISK_EXACT_ASYMP_ORDER}.
\end{proof}
 
\section*{Acknowledgements}
 The authors would like to thank Professor Jayanta Kumar Ghosh for letting them aware about the work of \citet{PKV2014} on the posterior contraction properties of the horseshoe prior.

\bibliographystyle{apalike}
\bibliography{reference_contraction}

\end{document}